\documentclass[letterpaper]{article}

\usepackage{amsmath}
\usepackage{amssymb}
\usepackage{amsthm}
\usepackage{tikz-cd}
\usepackage{mathrsfs}
\usepackage{hyperref}
\usepackage{enumerate}
\usepackage{todonotes}
\usepackage{enumitem}
\setlist[1]{itemsep=0pt}

\theoremstyle{plain}
\newtheorem{theorem}{Theorem}[section]
\newtheorem{lemma}[theorem]{Lemma}

\newtheorem{proposition}[theorem]{Proposition}

\newtheorem{conjecture}[theorem]{Conjecture}
\theoremstyle{definition}
\newtheorem{definition}[theorem]{Definition}

\newtheorem{remark}[theorem]{Remark}

\newtheoremstyle{customname}{}{}{\itshape}{}{\bfseries}{.}{.5em}{\thmnote{#3}}
\theoremstyle{customname}

\newcommand{\textdef}{\textbf}

\newcommand{\TAA}{
  \coordinate (y) at (0,0);
  \coordinate (x) at (0,1);
  \coordinate (a) at (-1,1);
  \coordinate (c) at (-1,0);
  \coordinate (b) at (1,1);
  \coordinate (d) at (1,0);
  \coordinate (e) at (2,1);
  \coordinate (f) at (2,0);
}
\newcommand{\TAAMA}{
  \coordinate (a) at (-1,1);
  \coordinate (c) at (-1,0);
  \coordinate (b) at (1,1);
  \coordinate (d) at (1,0);
  \coordinate (e) at (2,1);
  \coordinate (f) at (2,0);
}
\newcommand{\TAAMB}{
  \coordinate (a) at (-1,1);
  \coordinate (c) at (-1,0);
  \coordinate (e) at (2,1);
  \coordinate (f) at (2,0);
}

\newcommand{\TAB}{
  \coordinate (y) at (0,0);
  \coordinate (x) at (0,1);
  \coordinate (a) at (-1,1);
  \coordinate (c) at (-1,0);
  \coordinate (b) at (1,1);
  \coordinate (d) at (1,0);
}
\newcommand{\TABM}{
  \coordinate (a) at (-1,1);
  \coordinate (c) at (-1,0);
  \coordinate (b) at (1,1);
  \coordinate (d) at (1,0);
}

\newcommand{\TBA}{
      \coordinate (x) at (-0.5,0.5);
      \coordinate (y) at (0.5,0.5);
      \coordinate (w) at (0.5,-0.5);
      \coordinate (z) at (-0.5,-0.5);
      \coordinate (a) at (-1.35,0);
      \coordinate (b) at (1.35,0);
      \coordinate (e) at (-2.35,0);
      \coordinate (f) at (2.35,0);
}
\newcommand{\TBAM}{
  \coordinate (e) at (-2.35,0);
  \coordinate (f) at (2.35,0);
}

\newcommand{\TBB}{
  \coordinate (a) at (0,1);
  \coordinate (x) at (-0.3,0);
  \coordinate (y) at (-0.85,-0.5);
  \coordinate (z) at (0.3,0);
  \coordinate (w) at (0.85,-0.5);
  \coordinate (b) at (-1.7,-1);
  \coordinate (d) at (1.7,-1);
  \coordinate (e) at (0,2);
}
\newcommand{\TBBM}{
  \coordinate (a) at (0,0);
  \coordinate (b) at (-1.7,-1);
  \coordinate (d) at (1.7,-1);
  \coordinate (e) at (0,2);
}

\newcommand{\TBC}{
  \coordinate (x) at (0,0);
  \coordinate (y) at (30:0.5);
  \coordinate (z) at (-30:1);
  \coordinate (w) at (-90:0.5);
  \coordinate (d) at (210:1);
  \coordinate (a) at (150:0.5);
  \coordinate (b) at (90:1);
  \coordinate (c) at (-30:2);
  \coordinate (e) at (90:2);
  \coordinate (f) at (210:2);
}
\newcommand{\TBCA}{
  \coordinate (x) at (0.5,0);
  \coordinate (y) at (0.5,1);
  \coordinate (z) at (1.5,0);
  \coordinate (w) at (0.5,-1);
  \coordinate (a) at (-0.5,0);
  \coordinate (b) at (-0.5,1);
  \coordinate (c) at (2.5,0);
  \coordinate (d) at (-0.5,-1);
  \coordinate (e) at (-1.5,0);
  \coordinate (g) at (-2.5,0);
}
\newcommand{\TBCAM}{
  \coordinate (c) at (2.5,0);
  \coordinate (g) at (-2.5,0);
}
\newcommand{\TBCBM}{
  \coordinate (c) at (-30:2);
  \coordinate (e) at (90:2);
  \coordinate (f) at (210:2);
  \coordinate (b) at (0:0);
}

\newcommand{\TBD}{
  \coordinate (a) at (-1.5,0.5);
  \coordinate (b) at (1.5,0.5);
  \coordinate (c) at (1.5,-0.5);
  \coordinate (d) at (-1.5,-0.5);
  \coordinate (x) at (-0.5,0.5);
  \coordinate (y) at (0.5,0.5);
  \coordinate (z) at (0.5,-0.5);
  \coordinate (w) at (-0.5,-0.5);
}
\newcommand{\TBDMA}{
  \coordinate (a) at (-1.5,0.5);
  \coordinate (b) at (1.5,0.5);
  \coordinate (c) at (1.5,-0.5);
  \coordinate (d) at (-1.5,-0.5);
  \coordinate (y) at (0.5,0.5);
  \coordinate (z) at (0.5,-0.5);
}
\newcommand{\TBDMB}{
  \coordinate (a) at (-1.5,0.5);
  \coordinate (b) at (1.5,0.5);
  \coordinate (c) at (1.5,-0.5);
  \coordinate (d) at (-1.5,-0.5);
}

\newcommand{\TC}{
  \coordinate (x) at (0:0);
  \coordinate (y) at (90:1);
  \coordinate (z) at (-30:1);
  \coordinate (w) at (-150:1);
  \coordinate (a) at (120:1.732);
  \coordinate (b) at (60:1.732);
  \coordinate (c) at (0:1.732);
  \coordinate (d) at (-60:1.732);
  \coordinate (e) at (-120:1.732);
  \coordinate (f) at (-180:1.732);
}
\newcommand{\TCM}{
  \coordinate (x) at (0:0);
  \coordinate (y) at (90:0.8);
  \coordinate (z) at (-30:0.8);
  \coordinate (w) at (-150:0.8);
  \coordinate (a) at (120:1.732);
  \coordinate (b) at (60:1.732);
  \coordinate (c) at (0:1.732);
  \coordinate (d) at (-60:1.732);
  \coordinate (e) at (-120:1.732);
  \coordinate (f) at (-180:1.732);
}
 
\allowdisplaybreaks[1]

\DeclareMathOperator{\Fac}{Fac}
\DeclareMathOperator{\PerMat}{PerMat}

\title{Connected cubic graphs with the maximum number of perfect matchings}
\author{
Peter Horak\footnote{School of Interdisciplinary Arts and Sciences, University of Washington, Tacoma. Email: \texttt{horak@uw.edu}.} \and 
Dongryul Kim\footnote{Corresponding author, Department of Mathematics, Stanford University, 450 Jane Stanford Way, Building 380, Stanford, CA 94305. Email: \texttt{dkim04@stanford.edu}.}
}

\begin{document}

\maketitle

\begin{abstract}
  It is proved that for $n \ge 6$, the number of perfect matchings in a simple
  connected cubic graph on $2n$ vertices is at most $4f_{n-1}$, with $f_{n}$
  being the $n$-th Fibonacci number. The unique extremal graph is characterized
  as well.

  In addition, it is shown that the number of perfect matchings in any cubic
  graph $G$ equals the expected value of a random variable defined on all
  $2$-colorings of edges of $G$. Finally, an improved lower bound on the
  maximum number of cycles in a cubic graph is provided. 
\end{abstract}


\tableofcontents

\section{Introduction}

It is natural to ask how many cycles, Hamiltonian cycles, and $2$-factors a
graph can have. To the best of our knowledge this question for cycles was
considered for the first time by Ahrens \cite{Ahr97} in 1897. This paper
focuses on the number of $2$-factors in cubic graphs.

We first note that in a cubic graph, the complement of a $1$-factor (i.e., a
perfect matching) is a $2$-factor, and vice versa. Therefore, for a cubic graph
$G$, the number of perfect matchings, the number of $1$-factors, the number of
$2$-factors are all equal, and we denote this by
\[
  \PerMat(G) = \Fac(G). 
\]

There is an extensive literature on the number of perfect matchings in cubic
graphs. We mention here only the most pertinent results. It is shown in
\cite{EKKKN11} that the number of perfect matchings in a $2$-connected cubic
graph is at least exponential with its order, thus confirming an old conjecture
of Lov\'{a}sz and Plummer \cite{LP09}.

As for the maximum number of perfect matchings, Alon and Friedland \cite{AF08}
proved a general result. Its restriction to cubic graphs states:

\begin{theorem}[Alon--Friedland \cite{AF08}, 2008] \label{thm:Alon-Friedland}
  For a simple cubic graph $G$ on $2n$ vertices,
  \[
    \PerMat(G) \leq 6^{n/3}.
  \]
  This bound is tight, and it is attained by taking the disjoint union of
  bipartite complete graphs $K_{3,3}$.
\end{theorem}

In other words, the above theorem says that the complete bipartite graph
$K_{3,3}$ has the highest ``density'' of perfect matchings among all cubic
graphs; thus the disjoint union of its copies constitutes the extremal graph.
However, this result does not provide any insight into the structure of
extremal connected cubic graphs. 

\begin{remark}
  When $G$ is bipartite, Theorem~\ref{thm:Alon-Friedland} is an immediate
  consequence of the Br\`{e}gman--Minc inequality, proved by Br\`{e}gman
  \cite{Bre73} in 1973. The contribution of Alon and Friedland was introducing
  a clever trick that extended the result from bipartite graphs to all graphs.
  We later use the same trick, in the form of
  Lemma~\ref{lem:doubling-the-graph}. 
\end{remark}

If only connected cubic (but not necessarily simple) graphs are considered then
Galbiati \cite{Gal81} proved the following theorem. 

\begin{theorem}[Galbati \cite{Gal81}, 1981]
  For a connected cubic (multi) graph $G$ on $2n$ vertices,
  \[
    \PerMat(G) \leq 2^{n}+1.
  \]
  This bound is tight, and it is attained by taking a cycle of length $2n$ and
  putting parallel edges alternatively.
\end{theorem}

The main result of this paper characterizes the extremal graphs if only simple
connected cubic graphs are taken into account. It is somewhat counter-intuitive
that the extremal graphs for large enough $n$ are $2$-connected but not
$3$-connected. 

\begin{theorem} \label{AA}
  Let $G$ be a simple connected cubic graph on $2n$ vertices. Then, $\PerMat(G)
  \le m_n$, and moreover, equality is achieved if and only if $G$ is isomorphic
  to one of the graphs in the family $M_n$ (see
  Section~\ref{sec:maximal-graphs} for the definition of $m_n$ and $M_n$). In
  particular, for $n \ge 6$ we have a sharp bound 
  \[
    \PerMat(G) \le 4 f_{n-1},
  \]
  where $f_n$ denotes the $n$th Fibonacci number. 
\end{theorem}

For $3$-connected graphs, we make the following conjecture about the maximum number of perfect matchings. 

\begin{conjecture}
  Let $G$ be a simple $3$-connected cubic graph on $2n$ vertices, where $n \ge 3$. Then
  \[
    \PerMat(G) \le f_n + 2 f_{n-1} + 2,
  \]
  with equality attained if and only if $G$ is isomorphic to 
  \begin{itemize}
    \item the circular ladder graph (with vertices $x_1, \dotsc, x_n, y_1, \dotsc, y_n$
      and edges $x_i y_i, x_i x_{i+1}, y_i y_{i+1}$ for $1 \le i \le n$,
    indices taken modulo $n$) when $n$ even, and
    \item the M\"{o}bius ladder graph (with vertices $x_1, \dotsc, x_{2n}$ and
      edges $x_i x_{i+1}$ and $x_i x_{i+n}$ for $1 \le i \le 2n$, indices taken
    modulo $2n$) when $n$ is odd,
  \end{itemize}
  where $f_n$ denotes the $n$th Fibonacci number. 
\end{conjecture}

The conjecture was verified by a computer search for $n \le 9$. Note that the
difference between the $3$-connected case and the general case is by a small
constant. Unlike the maximal graphs, the graphs that minimize the number of
perfect matchings (under appropriate connectivity hypotheses) tend to be
complicated. 

In addition, we derive a formula for counting the number of perfect matchings. 

\begin{theorem} \label{BB}
  For a cubic graph $G$ of order $2n$, the number of perfect matchings can be
  calculated as the expected value $\PerMat(G)= \operatorname{E}(X)$, where $X$
  is a random variable defined on the set of all $2$-colorings $c$ on the edges
  of $G$, each coloring equally likely, and $X(c)=(-3)^{m_{c}}$, where $m_{c}$
  is the number of vertices of $G$ incident in $c$ with three edges of the same
  color.
\end{theorem}

The above formula is not feasible for practical calculations but we believe it
is of a theoretical value. The proof is based on interpreting the number of
perfect matchings as an evaluation of a suitable quantum field theory. We
include only a sketch of the proof in Section~\ref{sec:quantum-field-theory}. 

Finally, we turn to the original question of Ahrens \cite{Ahr97}, on the
maximum number of cycles in a connected graph $G = (V,E)$. It is most
convenient to study the number of cycles in a connected graph $G$ with respect
to its cyclomatic number $r(G) = \lvert E \rvert - \lvert V \rvert + 1$, since
the number of cycles is bounded by $2^{r(G)}-1$. Let $\Psi(r)$ be the maximum
number of cycles among all graphs with the cyclomatic number equal to $r$.
Entringer and Slater \cite{ES81} showed that the problem of determining
$\Psi(r)$ can be reduced to cubic graphs; they proved that, for $r \geq 3$,
there is a connected cubic graph with the cyclomatic number $r$ and $\Psi(r)$
cycles. In addition they conjectured that $\Psi(r) \sim 2^{r-1}$. So far the
best upper bound has been provided by Aldred and Thomassen \cite{AT08} who
proved there that $\Psi(r) \leq \frac{15}{16}2^{r}$.

As for the lower bounds on $\Psi (r)$, the first lower bound $\Psi (r)\geq
2^{r-1}+f(r)$, where the error function $f(r)$ is exponential has been given in
\cite{Hor14}. In Section~\ref{sec:many-cycles}, we show that a graph obtained
from $ C_{n} \times K_{2}$ by replacing consecutive pairs of ``parallel'' rungs
by ``crossing'' ones provides a bound on $\Psi (r)$ with an improved
exponential error term.

\begin{theorem} \label{CC}
  There exists a constant $c > 0$ for which 
  \[
    \Psi(r) \geq 2^{r-1} + c r 2^{r/2}
  \]
  for all sufficiently large $r$. 
\end{theorem}

For a more precise bound, we refer the reader to Section~\ref{sec:many-cycles}. 

\section{The maximal graphs} \label{sec:maximal-graphs}

Before we prove Theorem~\ref{AA}, we collect properties that we will later use
to prove Theorem~\ref{AA}. The results in this section mostly follow from a
routine verification, hence we leave some of the details to the reader.

We define the Fibonacci numbers by
\[
  f_n = \frac{1}{\sqrt{5}} (\varphi^n - \varphi^{-n}), \quad \varphi =
  \frac{1+\sqrt{5}}{2},
\]
so that they satisfy
\[
  f_0 = 0, \quad f_1 = 1, \quad f_2 = 1, \quad f_n = f_{n-1} + f_{n-2}
\]
for all integers $n$. 

For the purpose of stating and proving Theorem~\ref{AA}, we define $m_n$ for $n \ge 3$ as
\[
  m_2 = 3, \quad m_3 = 6, \quad m_4 = 9, \quad m_5 = 13, \quad m_n = 4 f_{n-1} \text{ for } n \ge 6.
\]
We collect the inequalities among $m_n$ that we will use later. Note that we asymptotically have $m_n \sim c \varphi^n$ where $c \approx 1.106$ and $\varphi \approx 1.618$ is the golden ratio; this is useful for doing a quick sanity check. 

\begin{lemma} \label{lem:inequalities-for-m}\quad 
  \begin{enumerate}[label=(\roman*)]
    \item \label{itm:ineq-viii} For $n \ge 8$, we have $m_n = m_{n-1} + m_{n-2}$. 
    \item \label{itm:ineq-i} For $n \ge 6$, we have $\frac{3}{2} m_{n-1} < m_n$. 
    \item \label{itm:ineq-ii} For $n \ge 5$, we have $2 m_{n-2} < m_{n}$. 
    \item \label{itm:ineq-iii} For $n \ge 8$, we have $4 m_{n-3} \le m_{n}$, with equality only for $n = 8$. 
    \item \label{itm:ineq-iv} For $n \ge 9$, we have $6 m_{n-4} < m_{n}$. 
    \item \label{itm:ineq-ix} For $n \ge 6$, we have $m_{n+1} < \sqrt{3} m_n$. 
    \item \label{itm:ineq-v} For $n \ge 3$, we have $m_{2n} < m_n^2$. 
    \item \label{itm:ineq-vii} For $a, b \ge 3$, if $a+b \ge 8$ then $m_a m_b < m_{a+b+1}$. 
  \end{enumerate}
\end{lemma}

\begin{proof}
  \ref{itm:ineq-viii} follows from the definition that $m_n = 4 f_{n-1}$ for $n \ge 6$. For \ref{itm:ineq-i}, we can check the inequality by hand when $6 \le n \le 8$. When $n \ge 9$, we use induction; assuming that the inequality holds for smaller $n \ge 6$, we can write
  \[
    \frac{3}{2} m_{n-1} = \frac{3}{2} m_{n-2} + \frac{3}{2} m_{n-3} < m_{n-1} + m_{n-2} = m_n
  \]
  by \ref{itm:ineq-viii}, since $n \ge 9$. The other inequalities \ref{itm:ineq-ii}, \ref{itm:ineq-iii}, \ref{itm:ineq-iv}, \ref{itm:ineq-ix} can be proved using a similar inductive argument. 

  For \ref{itm:ineq-v}, we first do the case $n \le 5$ manually. For $n \ge 6$, we note that
  \[
    m_{2n} = 4f_{2n-1} = 4 (f_n^2 + f_{n-1}^2) = \frac{1}{4}(m_{n+1}^2 + m_n^2) < \frac{1}{4} ((\sqrt{3} m_n)^2 + m_n^2) = m_n^2
  \]
  by \ref{itm:ineq-ix}. Here, the identity $f_{2n-1} = f_n^2 + f_{n-1}^2$ is folklore. 

  For \ref{itm:ineq-vii}, we induct on $a + b$. When $\max(a,b) \le 7$, we can
  check manually. If $\max(a,b) \ge 8$, without loss of generally assume that
  $a \ge 8$. Then by \ref{itm:ineq-viii} and the inductive hypothesis, 
  \[
    m_a m_b = m_{a-1} m_b + m_{a-2} m_b < m_{a+b} + m_{a+b-1} = m_{a+b+1}.
  \]
  Here, the induction hypothesis applies because $(a-2) + b \ge 8 - 2 + 3 = 9$. 
\end{proof}

We now define the graphs that attain the maximal number of $2$-factors for a
given number of vertices. 

\begin{definition}
  For $n \ge 2$, we define the family $M_n$ of simple connected cubic graph with $2n$ vertices as follows:
  \begin{itemize}
    \item $M_2 = \{ K_4 \}$;
    \item $M_3 = \{ K_{3,3} \}$;
    \item $M_4$ contains a single graph that is $K_{4,4}$ with a perfect matching removed;
    \item $M_5$ contains a single graph that is the M\"{o}bius ladder on $10$ vertices, see Figure~\ref{fig:M5};
    \item $M_6$ contains two graphs, one that is the circular ladder graph on $12$ vertices, see Figure~\ref{fig:M6}, and one that is the ladder graph with $K_{3,3}$ inserted at both ends, see Figure~\ref{fig:Mn},
    \item for $n \ge 7$, $M_n$ contains a single graph that is the ladder graph with $K_{3,3}$ inserted at both ends, see Figure~\ref{fig:Mn}. 
  \end{itemize}
  Except for $n = 6$, the family $M_n$ contains a single graph. We shall abuse notation by also writing $M_n$ for the unique graph contained the family $M_n$, if no confusion is likely to arise. 
\end{definition}

\begin{figure}[t]
  \centering
  \begin{tikzpicture}
    \def\r{0.4}
    \draw (0,0) arc (90:270:\r);
    \draw (0,1.5) arc (270:90:\r);
    \draw (0,0) to[out=0,in=-105] (0.5,0.75) to[out=75,in=180] (1,1.5);
    \draw (0,1.5) to[out=0,in=105] (0.5,0.75) to[out=-75,in=180] (1,0);
    \draw (1,0) -- (6,0);
    \draw (1,1.5) -- (6,1.5);
    \foreach \x in {1,2,3,4,5} {
      \draw (\x+0.5,0) -- (\x+0.5,1.5);
      \fill (\x+0.5,0) circle (0.05);
      \fill (\x+0.5,1.5) circle (0.05);
    }
    \draw (6,0) arc (90:-90:\r);
    \draw (6,1.5) arc (-90:90:\r);
    \draw (6,-2*\r) -- (0,-2*\r);
    \draw (6,1.5+2*\r) -- (0,1.5+2*\r);
  \end{tikzpicture}
  \caption{The M\"{o}bius ladder on $10$ vertices}
  \label{fig:M5}
\end{figure}

\begin{figure}[t]
  \centering
  \begin{tikzpicture}
    \def\r{0.4}
    \draw (0,0) arc (90:270:\r);
    \draw (0,1.5) arc (270:90:\r);
    \draw (0,0) -- (6,0);
    \draw (0,1.5) -- (6,1.5);
    \foreach \x in {0,1,2,3,4,5} {
      \draw (\x+0.5,0) -- (\x+0.5,1.5);
      \fill (\x+0.5,0) circle (0.05);
      \fill (\x+0.5,1.5) circle (0.05);
    }
    \draw (6,0) arc (90:-90:\r);
    \draw (6,1.5) arc (-90:90:\r);
    \draw (6,-2*\r) -- (0,-2*\r);
    \draw (6,1.5+2*\r) -- (0,1.5+2*\r);
  \end{tikzpicture}
  \caption{The circular ladder on $12$ vertices}
  \label{fig:M6}
\end{figure}

\begin{figure}[t]
  \centering
  \begin{tikzpicture}
    \draw (0,0) -- (2.5,0);
    \draw (3.5,0) -- (5,0);
    \draw (0,1.5) -- (2.5,1.5);
    \draw (3.5,1.5) -- (5,1.5);
    \draw (0,0) -- (-1,0.2) -- (-1,1.3) -- (0,1.5);
    \draw (0,0) -- (-2,-0.3) -- (-2,1.8) -- (0,1.5);
    \draw (-1,0.2) -- (-2,1.8);
    \draw (-1,1.3) -- (-2,-0.3);
    \fill (0,0) circle (0.05);
    \fill (0,1.5) circle (0.05);
    \fill (-1,0.2) circle (0.05);
    \fill (-1,1.3) circle (0.05);
    \fill (-2,-0.3) circle (0.05);
    \fill (-2,1.8) circle (0.05);
    \begin{scope}[shift={(5,0)},xscale=-1]
      \draw (0,0) -- (-1,0.2) -- (-1,1.3) -- (0,1.5);
      \draw (0,0) -- (-2,-0.3) -- (-2,1.8) -- (0,1.5);
      \draw (-1,0.2) -- (-2,1.8);
      \draw (-1,1.3) -- (-2,-0.3);
      \fill (0,0) circle (0.05);
      \fill (0,1.5) circle (0.05);
      \fill (-1,0.2) circle (0.05);
      \fill (-1,1.3) circle (0.05);
      \fill (-2,-0.3) circle (0.05);
      \fill (-2,1.8) circle (0.05);
    \end{scope}
    \foreach \x in {1,2,4} {
      \draw (\x,0) -- (\x,1.5);
      \fill (\x,0) circle (0.05);
      \fill (\x,1.5) circle (0.05);
    }
    \node at (3,0) {$\cdots$};
    \node at (3,1.5) {$\cdots$};
  \end{tikzpicture}
  \caption{The ladder graph with $K_{3,3}$ inserted at both ends}
  \label{fig:Mn}
\end{figure}

\begin{proposition}
  For every $n \ge 2$ and every graph $G \in M_n$, we have $\Fac(G) = m_n$. 
\end{proposition}

\begin{proof}
  For $n \le 8$, we can count the number of $2$-factors explicitly, and check that the number of $2$-factors of $G$ is exactly $m_n$. For $n \ge 9$, we induct on $n$. It is shown in Section~3, proof of Theorem~\ref{AA}, Subcase~1-1 that
  \[
    \Fac(M_n) = \Fac(M_{n-1}) + \Fac(M_{n-2}).
  \]
  Then it follows from Lemma~\ref{lem:inequalities-for-m}, \ref{itm:ineq-viii} that $\Fac(M_n) = m_{n-1} + m_{n-2} = m_n$ by the inductive hypothesis. 
\end{proof}

\section{Proof of Theorem~\ref{AA}} \label{sec:main-theorem}

For the reader's convenience, the proof of Theorem~\ref{AA} will be presented
in terms of $2$-factors instead of perfect matchings. We recall that the
complement of a perfect matching in a cubic graph is $2$-factor.

The general strategy for proving Theorem~\ref{AA} is to first prove the theorem
for bipartite graphs. Then using a trick of Alon and Friedland \cite{AF08}, we
deduce the general case. We start with an auxiliary result. 

\begin{lemma} \label{lem:connected-bipartite-cubic}
  Let $G$ be a simple connected bipartite cubic graph. 
  \begin{enumerate}
    \item[(i)] The graph $G$ does not have a bridge. 
    \item[(ii)] For an edge $xy$ of $G$, let $y,a,b$ be the neighbors of $x$ and let $x,c,d$ be the neighbors of $y$. Then the six vertices $x, y, a, b, c, d$ are all distinct. 
    \item[(ii)] For every edge $xy$ of $G$, the induced subgraph on $V(G) \setminus \{x,y\}$ has at most two components. 
  \end{enumerate}
\end{lemma}

\begin{proof}
  (i) If it has a bridge, we may remove the edge and take a component of the resulting graph. This is a bipartite graph with one vertex of degree $2$ and all other vertices having degree $3$. Counting the number of edges, we see that it is both a multiple of $3$ and equal to $2$ modulo $3$, arriving at a contradiction. 

  (ii) Since $G$ is bipartite, we may color the vertices of $G$ in black and white so that each edges is incident with one white vertex and one black vertex. Without loss of generality, assume that $x$ is black and $y$ is white. Then $x,c,d$ are black and $y,a,b$ are white. Because $G$ is simple, the three vertices $x,c,d$ are distinct, and similarly, the three vertices $y,a,b$ are distinct as well. On the other hand, the two groups of three vertices have different color, hence all six vertices $x,y,a,b,c,d$ are distinct. 

  (iii) Denote by $G^\prime$ the induced subgraph on $V(G) \setminus \{x,y\}$ and assume that $G^\prime$ has at least three components. Let the neighbors of $x$ be $y, a, b$ and the neighbors of $y$ be $x, c, d$. Because $G$ was connected, each one of the components of $G^\prime$ contains at least one of the four vertices $a, b, c, d$. By part~(i), the edge $xy$ is not a bridge in $G$. This means that there is a path in $G^\prime$ connecting either $a$ or $b$ to either $c$ or $d$. Hence, without loss of generality, we may as well assume that $a, c$ are in the same component in $G^\prime$. Then $b$ is in a different component than $a, c, d$ in $G^\prime$, which means that the edge $xb$ must be a bridge in $G$. This contradicts part~(i). 
\end{proof}

\begin{definition}
  For a simple connected bipartite cubic graph $G$, we say that an edge $e = uv$ is a \textdef{ladder-bridge} if the induced subgraph on $V(G) \setminus \{u,v\}$ is disconnected, or equivalently, has two components. 
\end{definition}

We first prove Theorem~\ref{AA} for bipartite graphs.

\begin{theorem} \label{AA-for-bipartite}
  Let $G$ be a simple connected cubic \emph{bipartite} graph on $2n$ vertices (so automatically $n \ge 3$). Then $\PerMat(G) = \Fac(G) \le m_n$, and moreover, equality is achieved if and only if $G$ is isomorphic to one of the graphs in the family $M_n$. 
\end{theorem}

\begin{remark}
  For $n \le 8$, we content ourselves with using a computer to verify the
  statement. (There are $60$ simple connected bipartite cubic graphs with at
  most $16$ vertices in total.) However, it is possible to avoid using a
  computer at all. We can get away with verifying the statement by hand for
  $n\leq 6$ only (in total, there are only $9$ simple connected cubic bipartite
  graphs with at most $12$ vertices up to isomorphism), but then the proof of the theorem
  becomes quite longer as many small cases for $n= 7, 8$ have to be considered
  separately.
\end{remark}

The proof proceeds by induction on $n$. We deal with the inductive step by dividing into three large cases, and in each case, finding a way to bound the number of perfect matchings by those of smaller graphs. In Case~1, we examine the case when $G$ contains a ladder-bridge. This case takes care of many of the degeneracies that appear in the later cases. In Case~2, we consider when $G$ has a $4$-cycle, and in Case~3, we look at the case when $G$ does not have a $4$-cycle. 

\begin{proof}[Proof of Theorem~\ref{AA-for-bipartite}]
  We use induction on $n$. For $n \le 8$, we can perform a computer search on all the cubic bipartite graphs. Hence, let us assume that $n \ge 9$ and that the statement holds for all smaller $n$. 

  \textbf{Case 1.} The graph $G$ has a ladder-bridge $xy$. Let us write $V(G) \setminus \{x,y\} = A \cup B$, where $A$ and $B$ are the components. Note that if all neighbors of $x$ are in $A \cup \{y\}$, then one of the edges incident with $y$ is a bridge of $G$, contradicting Lemma~\ref{lem:connected-bipartite-cubic}. Hence one of the neighbors of $x$ is in $A$, one of them is in $B$, and similarly for $y$. Define $a, c \in A$ and $b, d \in B$ so that the neighbors of $x$ are $y, a, b$ and the neighbors of $y$ are $x, c, d$, as in Figure~\ref{fig:ladder-bridge}. 

  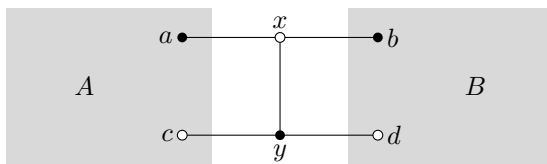
\begin{figure}[t]
    \centering
    \begin{tikzpicture}[scale=1.3]
      \coordinate (y) at (0,0);
      \coordinate (x) at (0,1);
      \coordinate (a) at (-1,1);
      \coordinate (c) at (-1,0);
      \coordinate (b) at (1,1);
      \coordinate (d) at (1,0);
      \fill[gray!30!white] (-2.8,-0.3) rectangle (-0.7,1.3);
      \fill[gray!30!white] (0.7,-0.3) rectangle (2.8,1.3);
      \draw (b) -- (x) -- (a);
      \draw (x) -- (y);
      \draw (c) -- (y) -- (d);
      \draw[fill=white] (x) circle (0.05);
      \fill (y) circle (0.05);
      \fill (a) circle (0.05);
      \fill (b) circle (0.05);
      \draw[fill=white] (c) circle (0.05);
      \draw[fill=white] (d) circle (0.05);
      \node[above] at (x) {$x$};
      \node[below] at (y) {$y$};
      \node[left] at (a) {$a$};
      \node[left] at (c) {$c$};
      \node[right] at (b) {$b$};
      \node[right] at (d) {$d$};
      \node at (-2,0.5) {$A$};
      \node at (2,0.5) {$B$};
    \end{tikzpicture}
    \caption{Case 1---a connected bipartite cubic graph $G$ with ladder-bridge $xy$}
    \label{fig:ladder-bridge}
  \end{figure}

  \textbf{Subcase 1-1.} Suppose either $a$ and $c$ are connected by an edge or $b$ and $d$ are connected by an edge. Without loss of generality, we assume that $b, d$ are connected by an edge. Let the neighbors of $b$ be $x, d, e$, and the neighbors of $d$ be $y, b, f$, as in Figure~\ref{fig:double-ladder-bridge}. We consider two other graphs: $G^\prime$, which is the graph obtained by removing the vertices $x, y$ and connecting $ad, cb$, and $G^{\prime\prime}$, which is the graph obtained by removing vertices $x, y, b, d$ and connecting $ae, cf$, see Figure~\ref{fig:double-ladder-bridge-resolution}. Note that both $G^\prime, G^{\prime\prime}$ are again simple connected bipartite cubic graphs, where connectivity of $G^{\prime\prime}$ follows from the fact that there is a path in $A$ connecting $a, c$. Any $2$-factor of $G$ will either contain both $ax, cy$ or contain neither, because these are the only edges connecting a vertex in $A$ and a vertex not in $A$. Using this, we see that there are five possible shapes a $2$-factor of $G$ can take on the induced subgraph on $\{ a,b,c,d,e,f,x,y \}$, listed on the leftmost column of Figure~\ref{fig:double-ladder-bridge-resolution}. Similarly, we can list all the possible shapes a $2$-factor of $G^\prime$ or $G^{\prime\prime}$ can have on the induced subgraph of $\{a,b,c,d,e,f\}$ or $\{a,c,e,f\}$, and these are depicted on the rightmost column of Figure~\ref{fig:double-ladder-bridge-resolution}. 

  \begin{figure}[t]
    \centering
    \begin{tikzpicture}[scale=1]
      \TAA
      \fill[gray!30!white] (-2.8,-0.3) rectangle (-0.7,1.3);
      \fill[gray!30!white] (0.7,-0.3) rectangle (3.3,1.3);
      \draw (e) -- (b) -- (x) -- (a);
      \draw (x) -- (y);
      \draw (b) -- (d);
      \draw (c) -- (y) -- (d) -- (f) ;
      \foreach \x in {(a),(b),(f),(y)} {
	\fill \x circle (0.05);
      }
      \foreach \x in {(c),(d),(e),(x)} {
	\draw[fill=white] \x circle (0.05);
      }
      \node[above] at (x) {$x$};
      \node[below] at (y) {$y$};
      \node[left] at (a) {$a$};
      \node[left] at (c) {$c$};
      \node[above] at (b) {$b$};
      \node[below] at (d) {$d$};
      \node[right] at (e) {$e$};
      \node[right] at (f) {$f$};
      \node at (-2,0.5) {$A$};
      \node at (2.5,0.5) {$B$};
    \end{tikzpicture}
    \caption{Subcase 1-1---when $xy$ is next to another ladder-bridge $bd$}
    \label{fig:double-ladder-bridge}
  \end{figure}
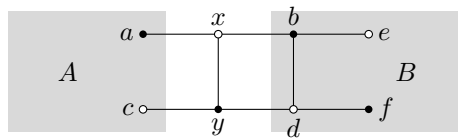

  We now modify the $2$-factor of $G$ to either a $2$-factor of $G^\prime$ or a $2$-factor of $G^{\prime\prime}$, by simply removing all the edges between $a,b,c,d,e,f,x,y$ and filling in that part with an appropriate diagram. If we do this process by taking the $i$th diagram on the leftmost column of Figure~\ref{fig:double-ladder-bridge-resolution} replacing with the $i$th diagram on the rightmost column of Figure~\ref{fig:double-ladder-bridge-resolution}. The process is reversible, as we can similarly take the right hand side diagram and replace it with the left hand side diagram. This shows that
  \[
    \Fac(G) = \Fac(G^\prime) + \Fac(G^{\prime\prime}).
  \]
  The inductive hypothesis applies to both $G^\prime$ and $G^{\prime\prime}$, as they are simple connected bipartite cubic graphs with number of vertices less than $2n$. Therefore
  \[
    \Fac(G) = \Fac(G^\prime) + \Fac(G^{\prime\prime}) \le m_{n-1} + m_{n-2} = m_n,
  \]
  where the last equality holds by Lemma~\ref{lem:inequalities-for-m}, \ref{itm:ineq-viii}, since $n \ge 9$. 
  
  When does equality hold? Again by the inductive hypothesis, equality holds if and only if $G^\prime$ is isomorphic to the graph $M_{n-1}$ and $G^{\prime\prime}$ is isomorphic to the graph $M_{n-2}$. In $G^\prime$, the edge $bd$ is a ladder-bridge. To recover $G$ from $G^\prime$, we need to insert another ladder-bridge, and it is easy to verify that the resulting graph is always isomorphic to $M_n$. 

  \begin{figure}[t]
    \centering
    \begin{tikzpicture}
      \begin{scope}[scale=0.7]
	\TAA
	\draw (e) -- (b) -- (x) -- (a);
	\draw (x) -- (y);
	\draw (b) -- (d);
	\draw (c) -- (y) -- (d) -- (f) ;
	\foreach \x in {(a),(b),(f),(y)} {
	  \fill \x circle (0.05);
	}
	\foreach \x in {(c),(d),(e),(x)} {
	  \draw[fill=white] \x circle (0.05);
	}
	\node at (0.5,-0.5) {$G$};
      \end{scope}
      \begin{scope}[shift={(4,1)},scale=0.7]
	\coordinate (a) at (-1,1);
	\coordinate (c) at (-1,0);
	\coordinate (b) at (1,1);
	\coordinate (d) at (1,0);
	\coordinate (e) at (2,1);
	\coordinate (f) at (2,0);
	\draw (e) -- (b) -- (c);
	\draw (b) -- (d);
	\draw (a) -- (d) -- (f) ;
	\foreach \x in {(a),(b),(f)} {
	  \fill \x circle (0.05);
	}
	\foreach \x in {(c),(d),(e)} {
	  \draw[fill=white] \x circle (0.05);
	}
	\node at (0.5,-0.5) {$G^\prime$};
      \end{scope}
      \begin{scope}[shift={(4,-1)},scale=0.7]
	\coordinate (a) at (-1,1);
	\coordinate (c) at (-1,0);
	\coordinate (e) at (2,1);
	\coordinate (f) at (2,0);
	\draw (e) -- (a);
	\draw (c) -- (f) ;
	\foreach \x in {(a),(f)} {
	  \fill \x circle (0.05);
	}
	\foreach \x in {(c),(e)} {
	  \draw[fill=white] \x circle (0.05);
	}
	\node at (0.5,-0.5) {$G^{\prime\prime}$};
      \end{scope}
      \draw[->] (1.8,0) -- (2.9,-0.3);
      \draw[->] (1.8,0.7) -- (2.9,1);

      \begin{scope}[shift={(-3,2.0875)},scale=0.5]
	\begin{scope}[shift={(0,0)}]
	  \TAA
	  \draw[thick] (a) -- (x) -- (y) -- (c);
	  \draw[thick] (f) -- (d) -- (b) -- (e);
	  \draw[dashed] (x) -- (b);
	  \draw[dashed] (d) -- (y);
	\end{scope}
	\begin{scope}[shift={(0,-2)}]
	  \TAA
	  \draw[thick] (a) -- (x) -- (b) -- (d) -- (y) -- (c);
	  \draw[dashed] (b) -- (e);
	  \draw[dashed] (x) -- (y);
	  \draw[dashed] (d) -- (f);
	\end{scope}
	\begin{scope}[shift={(0,-4)}]
	  \TAA
	  \draw[thick] (e) -- (b) -- (x) -- (y) -- (d) -- (f);
	  \draw[dashed] (b) -- (d);
	  \draw[dashed] (x) -- (a);
	  \draw[dashed] (c) -- (y);
	\end{scope}
	\begin{scope}[shift={(0,-6)}]
	  \TAA
	  \draw[thick] (a) -- (x) -- (b) -- (e);
	  \draw[thick] (c) -- (y) -- (d) -- (f);
	  \draw[dashed] (x) -- (y);
	  \draw[dashed] (b) -- (d);
	\end{scope}
	\begin{scope}[shift={(0,-8)}]
	  \TAA
	  \draw[thick] (x) -- (b) -- (d) -- (y) -- cycle;
	  \draw[dashed] (a) -- (x);
	  \draw[dashed] (c) -- (y);
	  \draw[dashed] (b) -- (e);
	  \draw[dashed] (d) -- (f);
	\end{scope}
      \end{scope}

      \begin{scope}[shift={(7.4,2.1)},scale=0.5]
	\begin{scope}[shift={(0,0)}]
	  \TAAMA
	  \draw[thick] (a) -- (d) -- (f);
	  \draw[thick] (c) -- (b) -- (e);
	  \draw[dashed] (b) -- (d);
	\end{scope}
	\begin{scope}[shift={(0,-2)}]
	  \TAAMA
	  \draw[thick] (a) -- (d) -- (b) -- (c);
	  \draw[dashed] (b) -- (e);
	  \draw[dashed] (d) -- (f);
	\end{scope}
	\begin{scope}[shift={(0,-4)}]
	  \TAAMA
	  \draw[thick] (e) -- (b) -- (d) -- (f);
	  \draw[dashed] (b) -- (c);
	  \draw[dashed] (a) -- (d);
	\end{scope}
	\begin{scope}[shift={(0,-6)}]
	  \TAAMB
	  \draw[thick] (a) -- (e);
	  \draw[thick] (c) -- (f);
	\end{scope}
	\begin{scope}[shift={(0,-8)}]
	  \TAAMB
	  \draw[dashed] (a) -- (e);
	  \draw[dashed] (c) -- (f);
	\end{scope}
      \end{scope}
    \end{tikzpicture}
    \caption{Subcase 1-1---modifying the graph $G$ to $G^\prime$ and $G^{\prime\prime}$}
    \label{fig:double-ladder-bridge-resolution}
  \end{figure}

  \textbf{Subcase 1-2.} Suppose now that there is no edge between $a, c$ and also between $b, d$. This time, we modify the graph $G$ to $G^\prime$ by removing the vertices $x,y$ and then connecting $a,c$ and $b,d$, see Figure~\ref{fig:isolated-ladder-bridge-resolution}. Then $G^\prime$ is a simple bipartite cubic graph, even though it is not connected. Using the same process of replacing the $i$th configuraiton of the leftmost column with the $i$th configuration of the rightmost column, from each $2$-factor of $G$ we get a $2$-factor of $G^\prime$. Moreover, it is clear that distinct $2$-factors of $G$ give distinct $2$-factors of $G^\prime$, even though some $2$-factors of $G^\prime$ do not appear by this process. Therefore $\Fac(G) \le \Fac(G^\prime)$. On the other hand, $G^\prime$ has two components, say $G^\prime = G_1^\prime \cup G_2^\prime$, where both $G_1^\prime, G_2^\prime$ are simple connected bipartite cubic graphs. Then the inductive hypothesis applies, so
  \[
    \Fac(G) \le \Fac(G^\prime) = \Fac(G_1^\prime) \Fac(G_2^\prime) \le m_a m_{n-a-1}
  \]
  where $G_1^\prime$ has $2a$ vertices and $G_2^\prime$ has $2(n-a-1)$ vertices. Because $n \ge 9$, from Lemma~\ref{lem:inequalities-for-m}, \ref{itm:ineq-vii}, we obtain
  \[
    \Fac(G) \le m_a m_{n-a-1} < m_n.
  \]

  \begin{figure}[t]
    \centering
    \begin{tikzpicture}
      \begin{scope}[scale=0.9]
	\TAB
	\draw (b) -- (x) -- (a);
	\draw (x) -- (y);
	\draw (c) -- (y) -- (d);
	\draw[fill=white] (x) circle (0.05);
	\fill (y) circle (0.05);
	\fill (a) circle (0.05);
	\fill (b) circle (0.05);
	\draw[fill=white] (c) circle (0.05);
	\draw[fill=white] (d) circle (0.05);
	\node at (0,-0.5) {$G$};
      \end{scope}
      \begin{scope}[shift={(4,0)},scale=0.9]
	\TABM
	\draw (c) arc [radius=0.5, start angle=-90, end angle=90];
	\draw (b) arc [radius=0.5, start angle=90, end angle=270];
	\fill (a) circle (0.05);
	\fill (b) circle (0.05);
	\draw[fill=white] (c) circle (0.05);
	\draw[fill=white] (d) circle (0.05);
	\node at (0,-0.5) {$G^\prime$};
      \end{scope}
      \draw[->] (1.5,0.45) -- (2.5,0.45);

      \begin{scope}[shift={(-3,1.5)},scale=0.7]
	\begin{scope}[shift={(0,0)}]
	  \TAB
	  \draw[thick] (a) -- (x) -- (b);
	  \draw[thick] (c) -- (y) -- (d);
	  \draw[dashed] (x) -- (y);
	\end{scope}
	\begin{scope}[shift={(0,-2)}]
	  \TAB
	  \draw[thick] (a) -- (x) -- (y) -- (c);
	  \draw[dashed] (x) -- (b);
	  \draw[dashed] (y) -- (d);
	\end{scope}
	\begin{scope}[shift={(0,-4)}]
	  \TAB
	  \draw[thick] (b) -- (x) -- (y) -- (d);
	  \draw[dashed] (c) -- (y);
	  \draw[dashed] (x) -- (a);
	\end{scope}
      \end{scope}

      \begin{scope}[shift={(7,1.5)},scale=0.7]
	\begin{scope}[shift={(0,0)}]
	  \TABM
	  \draw[thick] (c) arc [radius=0.5, start angle=-90, end angle=90];
	  \draw[thick] (b) arc [radius=0.5, start angle=90, end angle=270];
	\end{scope}
	\begin{scope}[shift={(0,-2)}]
	  \TABM
	  \draw[thick] (c) arc [radius=0.5, start angle=-90, end angle=90];
	  \draw[dashed] (b) arc [radius=0.5, start angle=90, end angle=270];
	\end{scope}
	\begin{scope}[shift={(0,-4)}]
	  \TABM
	  \draw[dashed] (c) arc [radius=0.5, start angle=-90, end angle=90];
	  \draw[thick] (b) arc [radius=0.5, start angle=90, end angle=270];
	\end{scope}
      \end{scope}
    \end{tikzpicture}
    \caption{Subcase 1-2---modifying the graph $G$ to $G^\prime$}
    \label{fig:isolated-ladder-bridge-resolution}
  \end{figure}

  \textbf{Case 2.} Now suppose there exists a $4$-cycle $xyzw$ in $G$. Denote by $a,y,w$ the neighbors of $x$, by $b,x,z$ the neighbors of $y$, by $c,y,w$ the neighbors of $z$, and by $d,x,z$ the neighbors of $w$, see Figure~\ref{fig:4-cycle}. By definition, $x,y,z,w$ are all distinct points, but there is no reason for $a,b,c,d$ to be all distinct. It is possible that $a = c$ or $b = d$ or both. 

  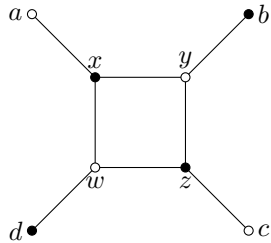
\begin{figure}[t]
    \centering
    \begin{tikzpicture}[scale=1.2]
      \coordinate (x) at (-0.5,0.5);
      \coordinate (y) at (0.5,0.5);
      \coordinate (z) at (0.5,-0.5);
      \coordinate (w) at (-0.5,-0.5);
      \coordinate (a) at (-1.2,1.2);
      \coordinate (b) at (1.2,1.2);
      \coordinate (c) at (1.2,-1.2);
      \coordinate (d) at (-1.2,-1.2);
      \draw (x) -- (y) -- (z) -- (w) -- cycle;
      \draw (a) -- (x);
      \draw (b) -- (y);
      \draw (c) -- (z);
      \draw (d) -- (w);
      \draw[fill=black] (x) circle (0.05);
      \draw[fill=white] (y) circle (0.05);
      \draw[fill=black] (z) circle (0.05);
      \draw[fill=white] (w) circle (0.05);
      \draw[fill=white] (a) circle (0.05);
      \draw[fill=black] (b) circle (0.05);
      \draw[fill=white] (c) circle (0.05);
      \draw[fill=black] (d) circle (0.05);
      \node[above] at (x) {$x$};
      \node[above] at (y) {$y$};
      \node[below] at (z) {$z$};
      \node[below] at (w) {$w$};
      \node[left] at (a) {$a$};
      \node[right] at (b) {$b$};
      \node[right] at (c) {$c$};
      \node[left] at (d) {$d$};
    \end{tikzpicture}
    \caption{Case 2---a cubic bipartite $G$ with a $4$-cycle $xyzw$}
    \label{fig:4-cycle}
  \end{figure}

  \textbf{Subcase 2-1.} First consider the case when $a = c$ and $b = d$. Now the vertices $a, b, x, y, z, w$ are all distinct. Moreover, $a$ cannot be connected to $b$ by an edge, otherwise we would have $V(G) = \{a,b,x,y,z,w\}$ and $G = K_{3,3}$, contradicting our assumption that $n \ge 9$. Denote by $x,z,e$ the neighbors of $a$ and $y,w,f$ the neighbors of $f$, see Figure~\ref{fig:4-cycle-K33-inserted}. As $a$ and $b$ are not neighbors, we see that all the vertices $a,b,e,f,x,y,z,w$ are distinct. If $e$ and $f$ are connected, then $ef$ becomes a ladder-bridge of $G$, hence this case is already covered in Case~1. Therefore we assume without loss of generality that $e$ and $f$ are not connected by an edge. 

  \begin{figure}[t]
    \centering
    \begin{tikzpicture}[scale=1.2]
      \TBA
      \draw (x) -- (y) -- (z) -- (w) -- cycle;
      \draw (e) -- (a) -- (x);
      \draw (a) -- (z);
      \draw (f) -- (b) -- (y);
      \draw (w) -- (b);
      \foreach \p in {(x), (z), (b), (e)} {
	\draw[fill=black] \p circle (0.05);
      }
      \foreach \p in {(y), (w), (a), (f)} {
	\draw[fill=white] \p circle (0.05);
      }
      \node[above] at (x) {$x$};
      \node[above] at (y) {$y$};
      \node[below] at (z) {$z$};
      \node[below] at (w) {$w$};
      \node[above] at (a) {$a$};
      \node[above] at (b) {$b$};
      \node[left] at (e) {$e$};
      \node[right] at (f) {$f$};
    \end{tikzpicture}
    \caption{Subcase 2-1---when $a=c$ and $b=d$}
    \label{fig:4-cycle-K33-inserted}
  \end{figure}
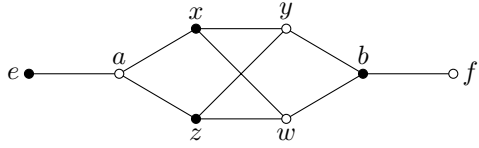

  Given such a graph $G$, we define a new graph $G^\prime$ by removing the vertices $a,b,x,y,z,w$ and then connecting $e$ and $f$. The resulting graph is simple, as $e$ and $f$ were not already connected, and also connected cubic bipartite. We again list the possible configurations of $2$-factors of $G$ restricted to this portion, of which there are $6$ as listed on the leftmost column of Figure~\ref{fig:4-cycle-K33-inserted-resolution}. We may again use the process of replacing the configurations of the leftmost column by configurations of the rightmost column. This time, the process is not injective, but every $2$-factor of $G^\prime$ can occur in at most $4$ different ways. This shows that
  \[
    \Fac(G) \le 4 \Fac(G^\prime) \le 4 m_{n-3} < m_n
  \]
  by Lemma~\ref{lem:inequalities-for-m}, \ref{itm:ineq-iii}. 

  \begin{figure}[t]
    \centering
    \begin{tikzpicture}[scale=0.9]
      \begin{scope}[scale=0.7]
	\TBA
	\draw (x) -- (y) -- (z) -- (w) -- cycle;
	\draw (e) -- (a) -- (x);
	\draw (a) -- (z);
	\draw (f) -- (b) -- (y);
	\draw (w) -- (b);
	\foreach \p in {(x), (z), (b), (e)} {
	  \draw[fill=black] \p circle (0.05);
	}
	\foreach \p in {(y), (w), (a), (f)} {
	  \draw[fill=white] \p circle (0.05);
	}
	\node at (0,-1) {$G$};
      \end{scope}
      \begin{scope}[shift={(5,0)},scale=0.7]
	\TBAM
	\draw (e) -- (f);
	\foreach \p in {(e)} {
	  \draw[fill=black] \p circle (0.05);
	}
	\foreach \p in {(f)} {
	  \draw[fill=white] \p circle (0.05);
	}
	\node at (0,-1) {$G^\prime$};
      \end{scope}
      \draw[->] (2,0) -- (3,0);

      \begin{scope}[shift={(-3,2.5)},scale=0.5]
	\begin{scope}[shift={(0,0)}]
	  \TBA
	  \draw[thick] (e) -- (a) -- (x) -- (w) -- (z) -- (y) -- (b) -- (f);
	  \draw[dashed] (a) -- (z);
	  \draw[dashed] (b) -- (w);
	  \draw[dashed] (x) -- (y);
	\end{scope}
	\begin{scope}[shift={(0,-2)}]
	  \TBA
	  \draw[thick] (e) -- (a) -- (x) -- (y) -- (z) -- (w) -- (b) -- (f);
	  \draw[dashed] (a) -- (z);
	  \draw[dashed] (b) -- (y);
	  \draw[dashed] (x) -- (w);
	\end{scope}
	\begin{scope}[shift={(0,-4)}]
	  \TBA
	  \draw[thick] (e) -- (a) -- (z) -- (w) -- (x) -- (y) -- (b) -- (f);
	  \draw[dashed] (a) -- (x);
	  \draw[dashed] (b) -- (w);
	  \draw[dashed] (z) -- (y);
	\end{scope}
	\begin{scope}[shift={(0,-6)}]
	  \TBA
	  \draw[thick] (e) -- (a) -- (z) -- (y) -- (x) -- (w) -- (b) -- (f);
	  \draw[dashed] (a) -- (x);
	  \draw[dashed] (b) -- (y);
	  \draw[dashed] (z) -- (w);
	\end{scope}
	\begin{scope}[shift={(0,-8)}]
	  \TBA
	  \draw[thick] (x) -- (y) -- (b) -- (w) -- (z) -- (a) -- cycle;
	  \draw[dashed] (a) -- (e);
	  \draw[dashed] (b) -- (f);
	  \draw[dashed] (x) -- (w);
	  \draw[dashed] (y) -- (z); 
	\end{scope}
	\begin{scope}[shift={(0,-10)}]
	  \TBA
	  \draw[thick] (x) -- (w) -- (b) -- (y) -- (z) -- (a) -- cycle;
	  \draw[dashed] (a) -- (e);
	  \draw[dashed] (b) -- (f);
	  \draw[dashed] (x) -- (y);
	  \draw[dashed] (w) -- (z); 
	\end{scope}
      \end{scope}

      \begin{scope}[shift={(8,2.5)},scale=0.5]
	\begin{scope}[shift={(0,0)}]
	  \TBAM
	  \draw[thick] (e) -- (f);
	\end{scope}
	\begin{scope}[shift={(0,-2)}]
	  \TBAM
	  \draw[thick] (e) -- (f);
	\end{scope}
	\begin{scope}[shift={(0,-4)}]
	  \TBAM
	  \draw[thick] (e) -- (f);
	\end{scope}
	\begin{scope}[shift={(0,-6)}]
	  \TBAM
	  \draw[thick] (e) -- (f);
	\end{scope}
	\begin{scope}[shift={(0,-8)}]
	  \TBAM
	  \draw[dashed] (e) -- (f);
	\end{scope}
	\begin{scope}[shift={(0,-10)}]
	  \TBAM
	  \draw[dashed] (e) -- (f);
	\end{scope}
      \end{scope}
    \end{tikzpicture}
    \caption{Subcase 2-1---modifying the graph $G$ to $G^\prime$}
    \label{fig:4-cycle-K33-inserted-resolution}
  \end{figure}
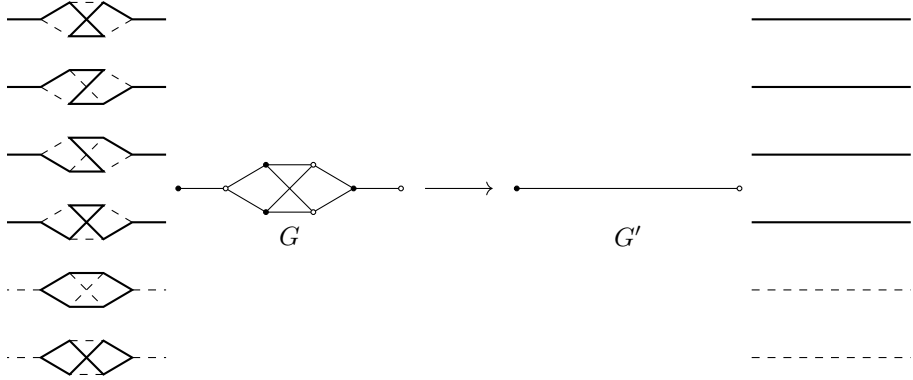

  \textbf{Subcase 2-2.} Now suppose that only one of the equalities from $a = c$ and $b = d$ hold. Without loss of generality, assume that $a = c$ and $b \neq d$. Denote by $x,z,e$ the neighbors of $a$, see Figure~\ref{fig:4-cycle-bipyramid}. It is clear that $a,x,y,z,w$ are all distinct, and also distinct from $b, d, e$. The only possible equalities between the points $a,b,d,e,x,y,z,w$ are $d = e$ or $b = e$. But if any of these equalities holds, we are reduced to Subcase~2-1. For instance, if $b = e$ then $axyz$ becomes a $4$-cycle satisfying the assumptions of Subcase~2-1. Therefore we may as well assume that all the vertices $a,b,d,e,x,y,z,w$ are distinct. 

  \begin{figure}[t]
    \centering
    \begin{tikzpicture}
      \TBB
      \draw (b) -- (y) -- (x) -- (a) -- (z) -- (w) -- (d);
      \draw (x) -- (w);
      \draw (z) -- (y);
      \draw (a) -- (e);
      \foreach \p in {(a),(y),(w)} {
	\draw[fill=white] \p circle (0.05);
      }
      \foreach \p in {(b),(d),(e),(x),(z)} {
	\draw[fill=black] \p circle (0.05);
      }
      \node[left] at (x) {$x$};
      \node[right] at (z) {$z$};
      \node[below] at (y) {$y$};
      \node[below] at (w) {$w$};
      \node[left] at (a) {$a$};
      \node[left] at (e) {$e$};
      \node[left] at (b) {$b$};
      \node[right] at (d) {$d$};
    \end{tikzpicture}
    \caption{Subcase 2-2---when $a = c$}
    \label{fig:4-cycle-bipyramid}
  \end{figure}
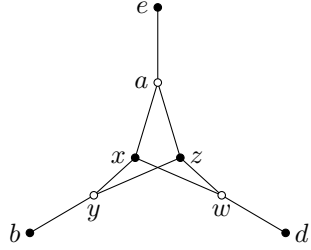

  We now construct a new graph $G^\prime$ by deleting the vertices $x,y,z,w$ and then connecting $a$ with $b$ and $d$. As $b,d,e$ are distinct points, the resulting graph is simple. That is, $G^\prime$ is a simple connected bipartite cubic graph, hence satisfies the induction hypothesis. We list the possible $2$-factors of $G$ and consider the replacing process as given by Figure~\ref{fig:4-cycle-bipyramid-resolution}. We observe that each $2$-factor of $G^\prime$ can occur in exactly $2$ ways, and therefore
  \[
    \Fac(G) = 2 \Fac(G^\prime) \le 2 m_{n-2} < m_n
  \]
  by Lemma~\ref{lem:inequalities-for-m}, \ref{itm:ineq-ii}. 

  \begin{figure}[t]
    \centering
    \begin{tikzpicture}
      \begin{scope}[scale=0.7]
	\TBB
	\draw (b) -- (y) -- (x) -- (a) -- (z) -- (w) -- (d);
	\draw (x) -- (w);
	\draw (z) -- (y);
	\draw (a) -- (e);
	\foreach \p in {(a),(y),(w)} {
	  \draw[fill=white] \p circle (0.05);
	}
	\foreach \p in {(b),(d),(e),(x),(z)} {
	  \draw[fill=black] \p circle (0.05);
	}
	\node at (0,-1) {$G$};
      \end{scope}
      \begin{scope}[shift={(3.5,0)},scale=0.7]
	\TBBM
	\draw (b) -- (a) -- (d);
	\draw (a) -- (e);
	\foreach \p in {(a)} {
	  \draw[fill=white] \p circle (0.05);
	}
	\foreach \p in {(b),(d),(e)} {
	  \draw[fill=black] \p circle (0.05);
	}
	\node at (0,-1) {$G^\prime$};
      \end{scope}
      \draw[->] (1.2,0) -- (2.3,0);

      \begin{scope}[shift={(-3.5,1.875)},scale=0.5]
	\begin{scope}[shift={(0,0)}]
	  \TBB
	  \draw[thick] (b) -- (y) -- (x) -- (a) -- (z) -- (w) -- (d);
	  \draw[dashed] (x) -- (w);
	  \draw[dashed] (y) -- (z);
	  \draw[dashed] (a) -- (e);
	\end{scope}
	\begin{scope}[shift={(2.6,-1.5)}]
	  \TBB
	  \draw[thick] (b) -- (y) -- (z) -- (a) -- (x) -- (w) -- (d);
	  \draw[dashed] (x) -- (y);
	  \draw[dashed] (w) -- (z);
	  \draw[dashed] (a) -- (e);
	\end{scope}
	\begin{scope}[shift={(0,-3)}]
	  \TBB
	  \draw[thick] (b) -- (y) -- (z) -- (w) -- (x) -- (a) -- (e);
	  \draw[dashed] (x) -- (y);
	  \draw[dashed] (a) -- (z);
	  \draw[dashed] (d) -- (w);
	\end{scope}
	\begin{scope}[shift={(2.6,-4.5)}]
	  \TBB
	  \draw[thick] (b) -- (y) -- (x) -- (w) -- (z) -- (a) -- (e);
	  \draw[dashed] (z) -- (y);
	  \draw[dashed] (a) -- (x);
	  \draw[dashed] (d) -- (w);
	\end{scope}
	\begin{scope}[shift={(0,-6)}]
	  \TBB
	  \draw[thick] (d) -- (w) -- (x) -- (y) -- (z) -- (a) -- (e);
	  \draw[dashed] (z) -- (w);
	  \draw[dashed] (a) -- (x);
	  \draw[dashed] (b) -- (y);
	\end{scope}
	\begin{scope}[shift={(2.6,-7.5)}]
	  \TBB
	  \draw[thick] (d) -- (w) -- (z) -- (y) -- (x) -- (a) -- (e);
	  \draw[dashed] (x) -- (w);
	  \draw[dashed] (a) -- (z);
	  \draw[dashed] (b) -- (y);
	\end{scope}
      \end{scope}

      \begin{scope}[shift={(5.7,1.875)},scale=0.5]
	\begin{scope}[shift={(0,0)}]
	  \TBBM
	  \draw[thick] (b) -- (a) -- (d);
	  \draw[dashed] (a) -- (e);
	\end{scope}
	\begin{scope}[shift={(2.6,-1.5)}]
	  \TBBM
	  \draw[thick] (b) -- (a) -- (d);
	  \draw[dashed] (a) -- (e);
	\end{scope}
	\begin{scope}[shift={(0,-3)}]
	  \TBBM
	  \draw[thick] (b) -- (a) -- (e);
	  \draw[dashed] (a) -- (d);
	\end{scope}
	\begin{scope}[shift={(2.6,-4.5)}]
	  \TBBM
	  \draw[thick] (b) -- (a) -- (e);
	  \draw[dashed] (a) -- (d);
	\end{scope}
	\begin{scope}[shift={(0,-6)}]
	  \TBBM
	  \draw[thick] (d) -- (a) -- (e);
	  \draw[dashed] (a) -- (b);
	\end{scope}
	\begin{scope}[shift={(2.6,-7.5)}]
	  \TBBM
	  \draw[thick] (d) -- (a) -- (e);
	  \draw[dashed] (a) -- (b);
	\end{scope}
      \end{scope}
    \end{tikzpicture}
    \caption{Subcase 2-2---modifying the graph $G$ to $G^\prime$}
    \label{fig:4-cycle-bipyramid-resolution}
  \end{figure}
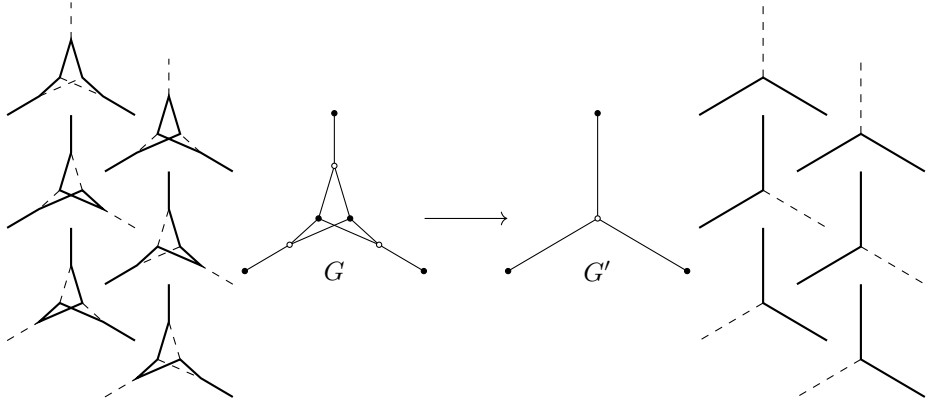

  \textbf{Subcase 2-3.} We are now left with the case when $a \neq c$ and $b \neq d$. In this Subcase, we consider the case when $\{ab, cd\} \cap E(G) \neq \emptyset$ and $\{bc, ad\} \cap E(G) \neq \emptyset$. Then without loss of generality, we can assume that $a$ is connected to both $b$ and $d$ by edges. Denote by $e,a,y$ the neighbors of $b$ and by $f,a,w$ the neighbors of $d$. The vertices $c, e, f$ cannot be all equal, because then we would have $V(G) = \{a,b,c,d,x,y,z,w\}$ which contradicts $n \ge 9$. Thus either the three vertices $c,e,f$ are all distinct, or two of them are equal and distinct from the last one. 

  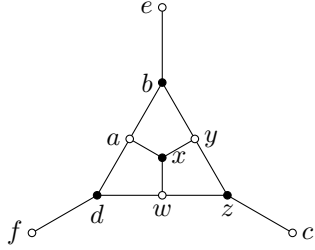
\begin{figure}[t]
    \centering
    \begin{tikzpicture}
      \TBC
      \draw (a) -- (x) -- (y) -- (z) -- (w) -- (d) -- (a) -- (b) -- (e);
      \draw (b) -- (y);
      \draw (d) -- (f);
      \draw (c) -- (z);
      \draw (x) -- (w);
      \foreach \p in {(a),(y),(w),(c),(e),(f)} {
	\draw[fill=white] \p circle (0.05);
      }
      \foreach \p in {(x),(z),(b),(d)} {
	\draw[fill=black] \p circle (0.05);
      }
      \node[right] at (x) {$x$};
      \node[right] at (y) {$y$};
      \node[below] at (z) {$z$};
      \node[below] at (w) {$w$};
      \node[left] at (a) {$a$};
      \node[left] at (b) {$b$};
      \node[right] at (c) {$c$};
      \node[below] at (d) {$d$};
      \node[left] at (e) {$e$};
      \node[left] at (f) {$f$};
    \end{tikzpicture}
    \caption{Subcase 2-3---when $a \neq c$ and $b \neq d$ but $a$ is connected to $b,d$}
    \label{fig:4-cycle-cube}
  \end{figure}

  \textbf{Subsubcase 2-3-1.} We first assume that two of $c,e,f$ are equal, but not all of them coincide. Here, by symmetry of Figure~\ref{fig:4-cycle-cube}, we may as well assume that $e = f$ but $c \neq e$. Let us denote by $g,b,d$ the neighbors of $e = f$. Then all the $10$ vertices $a,b,c,d,e,g,x,y,z,w$ are distinct. We now use the exact same strategy as Subcase~2-1. If $g$ and $c$ are connected by an edge, then the edge $cg$ becomes a ladder-bridge of $G$, hence we can deal with it using Case~1. If $g$ and $c$ are not connected by an edge, consider the graph $G^\prime$ obtained from $G$ by deleting the vertices $a,b,d,e,x,y,z,w$ and connecting $c$ and $g$ by an edge. Then $G^\prime$ is a simple connected bipartite cubic graph, hence the inductive hypothesis applies. From Figure~\ref{fig:4-cycle-cube-resolution1}, we see that
  \[
    \Fac(G) \le 6 \Fac(G^\prime) \le 6 m_{n-4} < m_n
  \]
  by Lemma~\ref{lem:inequalities-for-m}, \ref{itm:ineq-iv}. 

  \begin{figure}[t]
    \centering
    \begin{tikzpicture}
      \begin{scope}[scale=0.7]
	\TBCA
	\draw (c) -- (z) -- (w) -- (d) -- (e) -- (b) -- (y) -- (z);
	\draw (y) -- (x) -- (w);
	\draw (b) -- (a) -- (d);
	\draw (a) -- (x);
	\draw (g) -- (e);
	\foreach \p in {(a),(c),(e),(y),(w)} {
	  \draw[fill=white] \p circle (0.05);
	}
	\foreach \p in {(b),(d),(g),(x),(z)} {
	  \draw[fill=black] \p circle (0.05);
	}
	\node at (0,-1.5) {$G$};
      \end{scope}
      \begin{scope}[shift={(6.2,0)},scale=0.7]
	\TBCAM
	\draw (g) -- (c);
	\draw[fill=white] (c) circle (0.05);
	\draw[fill=black] (g) circle (0.05);
	\node at (0,-1.5) {$G^\prime$};
      \end{scope}
      \draw[->] (2.5,0) -- (3.7,0);

      \begin{scope}[shift={(0,-3)},scale=0.33]
	\begin{scope}[shift={(-6,3)}]
	  \TBCA
	  \draw[thick] (g) -- (e) -- (b) -- (a) -- (d) -- (w) -- (x) -- (y) -- (z) -- (c);
	  \draw[dashed] (e) -- (d);
	  \draw[dashed] (b) -- (y);
	  \draw[dashed] (a) -- (x);
	  \draw[dashed] (w) -- (z);
	\end{scope}
	\begin{scope}[shift={(0,3)}]
	  \TBCA
	  \draw[thick] (g) -- (e) -- (b) -- (y) -- (z) -- (c);
	  \draw[thick] (a) -- (x) -- (w) -- (d) -- cycle;
	  \draw[dashed] (e) -- (d);
	  \draw[dashed] (b) -- (a);
	  \draw[dashed] (y) -- (x);
	  \draw[dashed] (w) -- (z);
	\end{scope}
	\begin{scope}[shift={(6,3)}]
	  \TBCA
	  \draw[thick] (g) -- (e) -- (b) -- (y) -- (x) -- (a) -- (d) -- (w) -- (z) -- (c);
	  \draw[dashed] (e) -- (d);
	  \draw[dashed] (b) -- (a);
	  \draw[dashed] (w) -- (x);
	  \draw[dashed] (y) -- (z);
	\end{scope}
	\begin{scope}[shift={(-6,0)}]
	  \TBCA
	  \draw[thick] (g) -- (e) -- (d) -- (a) -- (b) -- (y) -- (x) -- (w) -- (z) -- (c);
	  \draw[dashed] (e) -- (b);
	  \draw[dashed] (x) -- (a);
	  \draw[dashed] (y) -- (z);
	  \draw[dashed] (w) -- (d);
	\end{scope}
	\begin{scope}[shift={(0,0)}]
	  \TBCA
	  \draw[thick] (g) -- (e) -- (d) -- (w) -- (z) -- (c);
	  \draw[thick] (a) -- (x) -- (y) -- (b) -- cycle;
	  \draw[dashed] (e) -- (b);
	  \draw[dashed] (d) -- (a);
	  \draw[dashed] (w) -- (x);
	  \draw[dashed] (y) -- (z);
	\end{scope}
	\begin{scope}[shift={(6,0)}]
	  \TBCA
	  \draw[thick] (g) -- (e) -- (d) -- (w) -- (x) -- (a) -- (b) -- (y) -- (z) -- (c);
	  \draw[dashed] (e) -- (b);
	  \draw[dashed] (d) -- (a);
	  \draw[dashed] (y) -- (x);
	  \draw[dashed] (w) -- (z);
	\end{scope}
	\begin{scope}[shift={(-6,-3)}]
	  \TBCA
	  \draw[thick] (a) -- (b) -- (e) -- (d) -- cycle;
	  \draw[thick] (x) -- (y) -- (z) -- (w) -- cycle;
	  \draw[dashed] (g) -- (e);
	  \draw[dashed] (c) -- (z);
	  \draw[dashed] (b) -- (y);
	  \draw[dashed] (a) -- (x);
	  \draw[dashed] (d) -- (w);
	\end{scope}
	\begin{scope}[shift={(0,-3)}]
	  \TBCA
	  \draw[thick] (e) -- (d) -- (a) -- (x) -- (w) -- (z) -- (y) -- (b) -- cycle;
	  \draw[dashed] (g) -- (e);
	  \draw[dashed] (c) -- (z);
	  \draw[dashed] (a) -- (b);
	  \draw[dashed] (x) -- (y);
	  \draw[dashed] (d) -- (w);
	\end{scope}
	\begin{scope}[shift={(6,-3)}]
	  \TBCA
	  \draw[thick] (e) -- (b) -- (a) -- (x) -- (y) -- (z) -- (w) -- (d) -- cycle;
	  \draw[dashed] (g) -- (e);
	  \draw[dashed] (c) -- (z);
	  \draw[dashed] (a) -- (d);
	  \draw[dashed] (x) -- (w);
	  \draw[dashed] (b) -- (y);
	\end{scope}
      \end{scope}

      \begin{scope}[shift={(6.2,-3)},scale=0.33]
	\foreach \x/\y in {-6/3, 0/3, 6/3, -6/0, 0/0, 6/0} {
	  \begin{scope}[shift={(\x,\y)}]
	    \TBCAM
	    \draw[thick] (g) -- (c);
	  \end{scope}
	}
	\foreach \x/\y in {-6/-3, 0/-3, 6/-3} {
	  \begin{scope}[shift={(\x,\y)}]
	    \TBCAM
	    \draw[dashed] (g) -- (c);
	  \end{scope}
	}
      \end{scope}
    \end{tikzpicture}
    \caption{Subsubcase 2-3-1---modifying the graph $G$ to $G^\prime$}
    \label{fig:4-cycle-cube-resolution1}
  \end{figure}

  \textbf{Subsubcase 2-3-2.} We now suppose that $c,e,f$ are all distinct. This time we follow Subcase~2-2. Consider the graph $G^\prime$ obtained by deleting the vertices $a,d,x,y,z,w$ from $G$ and then connecting $b$ to both $c, f$. Since $c,e,f$ are distinct points, the modified graph $G^\prime$ is simple, and also connected bipartite cubic. We can list the possible $2$-factors of $G$ restricted to the subgraph as in Figure~\ref{fig:4-cycle-cube-resolution2}. Now every $2$-factor of $G^\prime$ can be obtained from a $2$-factor of $G$ in exactly $3$ ways, hence
  \[
    \Fac(G) = 3 \Fac(G^\prime) \le 3 m_{n-3} < 4 m_{n-3} < m_n
  \]
  by Lemma~\ref{lem:inequalities-for-m}, \ref{itm:ineq-iii}. 

  \begin{figure}[t]
    \centering
    \begin{tikzpicture}
      \begin{scope}[scale=0.7]
	\TBC
	\draw (c) -- (z) -- (w) -- (d) -- (a) -- (b) -- (y) -- (z);
	\draw (e) -- (b);
	\draw (y) -- (x) -- (a);
	\draw (x) -- (w);
	\draw (d) -- (f);
	\foreach \p in {(a),(c),(e),(f),(y),(w)} {
	  \draw[fill=white] \p circle (0.05);
	}
	\foreach \p in {(b),(d),(x),(z)} {
	  \draw[fill=black] \p circle (0.05);
	}
	\node at (-90:1.2) {$G$};
      \end{scope}
      \begin{scope}[shift={(6.2,0)},scale=0.7]
	\TBCBM
	\draw (c) -- (b) -- (f);
	\draw (b) -- (e);
	\foreach \p in {(c),(e),(f)} {
	  \draw[fill=white] \p circle (0.05);
	}
	\draw[fill=black] (b) circle (0.05);
	\node at (-90:1.2) {$G^\prime$};
      \end{scope}
      \draw[->] (2.5,0) -- (3.7,0);

      \begin{scope}[shift={(0,-3.7)},scale=0.45]
	\begin{scope}[shift={(-4,3)}]
	  \TBC
	  \draw[thick] (e) -- (b) -- (a) -- (d) -- (f);
	  \draw[thick] (x) -- (y) -- (z) -- (w) -- cycle;
	  \draw[dashed] (b) -- (y);
	  \draw[dashed] (a) -- (x);
	  \draw[dashed] (d) -- (w);
	  \draw[dashed] (c) -- (z);
	\end{scope}
	\begin{scope}[shift={(0,3)}]
	  \TBC
	  \draw[thick] (e) -- (b) -- (a) -- (x) -- (y) -- (z) -- (w) -- (d) -- (f);
	  \draw[dashed] (b) -- (y);
	  \draw[dashed] (a) -- (d);
	  \draw[dashed] (x) -- (w);
	  \draw[dashed] (c) -- (z);
	\end{scope}
	\begin{scope}[shift={(4,3)}]
	  \TBC
	  \draw[thick] (e) -- (b) -- (y) -- (z) -- (w) -- (x) -- (a) -- (d) -- (f);
	  \draw[dashed] (b) -- (a);
	  \draw[dashed] (w) -- (d);
	  \draw[dashed] (x) -- (y);
	  \draw[dashed] (c) -- (z);
	\end{scope}
	\begin{scope}[shift={(-4,0)}]
	  \TBC
	  \draw[thick] (e) -- (b) -- (y) -- (z) -- (c);
	  \draw[thick] (x) -- (w) -- (d) -- (a) -- cycle;
	  \draw[dashed] (b) -- (a);
	  \draw[dashed] (y) -- (x);
	  \draw[dashed] (z) -- (w);
	  \draw[dashed] (d) -- (f);
	\end{scope}
	\begin{scope}[shift={(0,0)}]
	  \TBC
	  \draw[thick] (e) -- (b) -- (a) -- (d) -- (w) -- (x) -- (y) -- (z) -- (c);
	  \draw[dashed] (b) -- (y);
	  \draw[dashed] (a) -- (x);
	  \draw[dashed] (z) -- (w);
	  \draw[dashed] (d) -- (f);
	\end{scope}
	\begin{scope}[shift={(4,0)}]
	  \TBC
	  \draw[thick] (e) -- (b) -- (y) -- (x) -- (a) -- (d) -- (w) -- (z) -- (c);
	  \draw[dashed] (b) -- (a);
	  \draw[dashed] (w) -- (x);
	  \draw[dashed] (z) -- (y);
	  \draw[dashed] (d) -- (f);
	\end{scope}
	\begin{scope}[shift={(-4,-3)}]
	  \TBC
	  \draw[thick] (f) -- (d) -- (w) -- (z) -- (c);
	  \draw[thick] (x) -- (y) -- (b) -- (a) -- cycle;
	  \draw[dashed] (d) -- (a);
	  \draw[dashed] (w) -- (x);
	  \draw[dashed] (z) -- (y);
	  \draw[dashed] (b) -- (e);
	\end{scope}
	\begin{scope}[shift={(0,-3)}]
	  \TBC
	  \draw[thick] (f) -- (d) -- (a) -- (b) -- (y) -- (x) -- (w) -- (z) -- (c);
	  \draw[dashed] (z) -- (y);
	  \draw[dashed] (a) -- (x);
	  \draw[dashed] (d) -- (w);
	  \draw[dashed] (b) -- (e);
	\end{scope}
	\begin{scope}[shift={(4,-3)}]
	  \TBC
	  \draw[thick] (f) -- (d) -- (w) -- (x) -- (a) -- (b) -- (y) -- (z) -- (c);
	  \draw[dashed] (d) -- (a);
	  \draw[dashed] (w) -- (z);
	  \draw[dashed] (x) -- (y);
	  \draw[dashed] (b) -- (e);
	\end{scope}
      \end{scope}

      \begin{scope}[shift={(6.2,-3.7)},scale=0.45]
	\foreach \x/\y in {-4/3, 0/3, 4/3} {
	  \begin{scope}[shift={(\x,\y)}]
	    \TBCBM
	    \draw[thick] (e) -- (b) -- (f);
	    \draw[dashed] (b) -- (c);
	  \end{scope}
	}
	\foreach \x/\y in {-4/0, 0/0, 4/0} {
	  \begin{scope}[shift={(\x,\y)}]
	    \TBCBM
	    \draw[thick] (e) -- (b) -- (c);
	    \draw[dashed] (b) -- (f);
	  \end{scope}
	}
	\foreach \x/\y in {-4/-3, 0/-3, 4/-3} {
	  \begin{scope}[shift={(\x,\y)}]
	    \TBCBM
	    \draw[thick] (f) -- (b) -- (c);
	    \draw[dashed] (b) -- (e);
	  \end{scope}
	}
      \end{scope}
    \end{tikzpicture}
    \caption{Subsubcase 2-3-2---modifying the graph $G$ to $G^\prime$}
    \label{fig:4-cycle-cube-resolution2}
  \end{figure}

  \textbf{Subcase 2-4.} Finally, we left with the case when $a \neq c$, $b \neq d$, and either $\{ab, cd\} \cap E(G) = \emptyset$ or $\{bc, ad\} \cap E(G) = \emptyset$. Without loss of generality, suppose that $a,b$ are not connected by an edge and $c,d$ are also not connected by an edge. In this case, as in Subcase~1-1, we consider two graphs. Let $G^\prime$ be the graph obtained from $G$ by removing $x, w$ and then connecting $a, z$ and $y, d$. Let $G^{\prime\prime}$ be the graph obtained from $G$ by removing $x, y, z, w$ and then connecting $a,b$ and $c,d$. Since we have assumed that $a,b$ and $c,d$ are not already connected in $G$, we see that $G^\prime, G^{\prime\prime}$ are both simple graphs. It is clear that $G^\prime$ is connected. We may assume that $G^{\prime\prime}$ is also connected, because if it is not connected then removing $x, y$ from $G$ disconnects the graph. This would mean that $xy$ is a ladder-bridge of $G$, but then Case~1 handles the situation. Therefore we may suppose both $G^\prime$ and $G^{\prime\prime}$ are simple connected bipartite cubic graphs, and the inductive hypothesis applies. We now enumerate the posssible $2$-factors and do the replacement procedure according to Figure~\ref{fig:4-cycle-box-resolution}. From the usual analysis, it follows that
  \[
    \Fac(G) \le \Fac(G^\prime) + \Fac(G^{\prime\prime}) \le m_{n-1} + m_{n-2} = m_n
  \]
  by Lemma~\ref{lem:inequalities-for-m}, \ref{itm:ineq-viii}. 
  
  \begin{figure}[t]
    \centering
    \begin{tikzpicture}
      \begin{scope}[scale=0.7]
	\TBD
	\draw (a) -- (x) -- (y) -- (b);
	\draw (d) -- (w) -- (z) -- (c);
	\draw (x) -- (w);
	\draw (y) -- (z);
	\foreach \p in {(a),(c),(y),(w)} {
	  \draw[fill=white] \p circle (0.05);
	}
	\foreach \p in {(b),(d),(x),(z)} {
	  \draw[fill=black] \p circle (0.05);
	}
	\node at (0,-1) {$G$};
      \end{scope}
      \begin{scope}[shift={(5.2,0)},scale=0.7]
	\TBDMA
	\draw (a) -- (z) -- (c);
	\draw (y) -- (z);
	\draw (b) -- (y) -- (d) ;
	\foreach \x in {(d),(b),(z)} {
	  \draw[fill=black] \x circle (0.05);
	}
	\foreach \x in {(c),(a),(y)} {
	  \draw[fill=white] \x circle (0.05);
	}
	\node at (0,-1) {$G^\prime$};
      \end{scope}
      \begin{scope}[shift={(7.8,0)},scale=0.7]
	\TBDMB
	\draw (a) -- (b);
	\draw (c) -- (d) ;
	\foreach \x in {(b),(d)} {
	  \draw[fill=black] \x circle (0.05);
	}
	\foreach \x in {(c),(a)} {
	  \draw[fill=white] \x circle (0.05);
	}
	\node at (0,-1) {$G^{\prime\prime}$};
      \end{scope}
      \draw[->] (2,0) -- (3,0);

      \begin{scope}[shift={(0,-2.8)},scale=0.5]
	\begin{scope}[shift={(-4,2)}]
	  \TBD
	  \draw[thick] (a) -- (x) -- (y) -- (z) -- (w) -- (d);
	  \draw[dashed] (x) -- (w);
	  \draw[dashed] (b) -- (y);
	  \draw[dashed] (c) -- (z);
	\end{scope}
	\begin{scope}[shift={(-4,0)}]
	  \TBD
	  \draw[thick] (b) -- (y) -- (x) -- (w) -- (z) -- (c);
	  \draw[dashed] (a) -- (x);
	  \draw[dashed] (d) -- (w);
	  \draw[dashed] (y) -- (z);
	\end{scope}
	\begin{scope}[shift={(-4,-2)}]
	  \TBD
	  \draw[thick] (a) -- (x) -- (w) -- (d);
	  \draw[thick] (b) -- (y) -- (z) -- (c);
	  \draw[dashed] (x) -- (y);
	  \draw[dashed] (w) -- (z);
	\end{scope}
	\begin{scope}[shift={(0,1)}]
	  \TBD
	  \draw[thick] (a) -- (x) -- (y) -- (b);
	  \draw[thick] (d) -- (w) -- (z) -- (c);
	  \draw[dashed] (x) -- (w);
	  \draw[dashed] (y) -- (z);
	\end{scope}
	\begin{scope}[shift={(0,-1)}]
	  \TBD
	  \draw[thick] (a) -- (x) -- (w) -- (z) -- (y) -- (b);
	  \draw[dashed] (d) -- (w);
	  \draw[dashed] (x) -- (y);
	  \draw[dashed] (z) -- (c);
	\end{scope}
	\begin{scope}[shift={(4,1)}]
	  \TBD
	  \draw[thick] (d) -- (w) -- (x) -- (y) -- (z) -- (c);
	  \draw[dashed] (a) -- (x);
	  \draw[dashed] (w) -- (z);
	  \draw[dashed] (b) -- (y);
	\end{scope}
	\begin{scope}[shift={(4,-1)}]
	  \TBD
	  \draw[thick] (x) -- (y) -- (z) -- (w) -- cycle;
	  \draw[dashed] (a) -- (x);
	  \draw[dashed] (d) -- (w);
	  \draw[dashed] (b) -- (y);
	  \draw[dashed] (c) -- (z);
	\end{scope}
      \end{scope}

      \begin{scope}[shift={(6.5,-2.8)},scale=0.5]
	\begin{scope}[shift={(-4,2)}]
	  \TBDMA
	  \draw[thick] (a) -- (z) -- (y) -- (d);
	  \draw[dashed] (b) -- (y);
	  \draw[dashed] (c) -- (z);
	\end{scope}
	\begin{scope}[shift={(-4,0)}]
	  \TBDMA
	  \draw[thick] (b) -- (y) -- (z) -- (c);
	  \draw[dashed] (a) -- (z);
	  \draw[dashed] (d) -- (y);
	\end{scope}
	\begin{scope}[shift={(-4,-2)}]
	  \TBDMA
	  \draw[thick] (a) -- (z) -- (c);
	  \draw[thick] (b) -- (y) -- (d);
	  \draw[dashed] (z) -- (y);
	\end{scope}
	\begin{scope}[shift={(0,1)}]
	  \TBDMB
	  \draw[thick] (a) -- (b);
	  \draw[thick] (d) -- (c);
	\end{scope}
	\begin{scope}[shift={(0,-1)}]
	  \TBDMB
	  \draw[thick] (a) -- (b);
	  \draw[dashed] (d) -- (c);
	\end{scope}
	\begin{scope}[shift={(4,1)}]
	  \TBDMB
	  \draw[thick] (d) -- (c);
	  \draw[dashed] (b) -- (a);
	\end{scope}
	\begin{scope}[shift={(4,-1)}]
	  \TBDMB
	  \draw[dashed] (a) -- (b);
	  \draw[dashed] (d) -- (c);
	\end{scope}
      \end{scope}
    \end{tikzpicture}
    \caption{Subcase 2-4---modifying the graph $G$ to $G^\prime$ and $G^{\prime\prime}$}
    \label{fig:4-cycle-box-resolution}
  \end{figure}
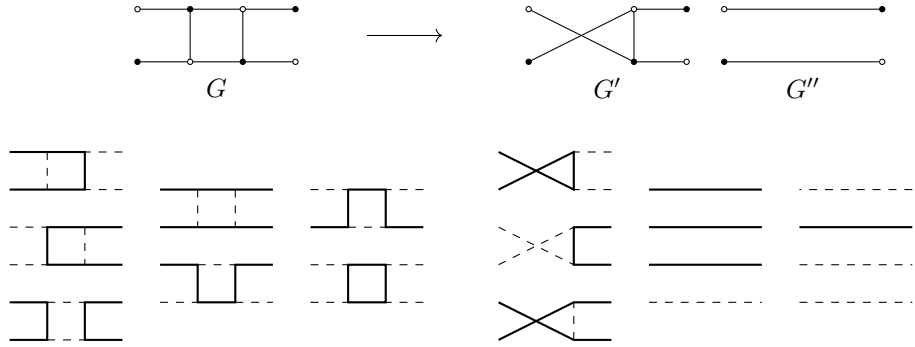

  We now analyze the equality case. Similarly to Subcase~1-1, by the inductive hypothesis, equality holds only if $G^\prime$ and $G^{\prime\prime}$ are isomorphic to $M_{n-1}$ and $M_{n-2}$. For the equality $\Fac(G) = \Fac(G^\prime) + \Fac(G^{\prime\prime})$ to hold, we further need that there is no $2$-factor of $G^\prime$ using the edges $az, zy, yb$ and also no $2$-factor using the edges $dy, yz, zc$. By inspection, we see that the only edges $e$ in $M_{n-1}$ satisfying the above property for $yz$ are precisely the ladder-bridges. This shows that $yz$ corresponds to a ladder-bridge in $M_{n-1}$, and modifying $G^\prime$ to $G$ shows that $G$ is isomorphic to the graph $M_n$. 

  \textbf{Case 3.} Since Case~2 was when $G$ has a $4$-cycle, we now assume that $G$ has no $4$-cycles. Since $G$ is bipartite, this implies that all cycles of $G$ has length at least $6$. Pick an arbitrary vertex $x$, denote its neighbors by $y,z,w$, denote the neighbors of $y$ by $x, a, b$, the neighbors of $z$ by $x, c, d$, and the neighbors of $w$ by $x, e, f$, as in Figure~\ref{fig:no-4-cycle}. As $G$ has no cycles of length smaller than $6$, we immediately see that the vertices $x, y, z, w, a, b, c, d, e, f$ are all distinct. 

  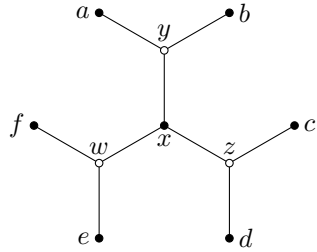
\begin{figure}[t]
    \centering
    \begin{tikzpicture}
      \TC
      \draw (c) -- (z) -- (x) -- (y) -- (a);
      \draw (y) -- (b);
      \draw (z) -- (d);
      \draw (x) -- (w) -- (e);
      \draw (w) -- (f);
      \foreach \p in {(x), (a), (b), (c), (d), (e), (f)} {
	\draw[fill=black] \p circle (0.05);
      }
      \foreach \p in {(y), (z), (w)} {
	\draw[fill=white] \p circle (0.05);
      }
      \node[below] at (x) {$x$};
      \node[above] at (y) {$y$};
      \node[above] at (z) {$z$};
      \node[above] at (w) {$w$};
      \node[left] at (a) {$a$};
      \node[right] at (b) {$b$};
      \node[right] at (c) {$c$};
      \node[right] at (d) {$d$};
      \node[left] at (e) {$e$};
      \node[left] at (f) {$f$};
    \end{tikzpicture}
    \caption{Case 3---when there is no $4$-cycle}
    \label{fig:no-4-cycle}
  \end{figure}

  We now consider six graphs constructed from $G$. The graph $G_j^y$ for $j = 1, 2$ are obtained from $G$ by removing the vertices $x, y$, and then adding the edges $aw$ and $bz$ for $j = 1$, and adding $az$ and $bw$ for $j = 2$. The other graphs $G_1^y, G_2^y, G_1^z, G_2^z$ are constructed similarly, as depicted in Figure~\ref{fig:no-4-cycle-modification}. It is clear that the graphs $G_j^y, G_j^z, G_j^w$ are all simple bipartite cubic graphs. We claim that we may assume that these graphs are also connected for $1 \le i \le 6$. For instance, suppose that $G_1^y$ is not connected. Then $G_1^y$ with the edges $aw$ and $bz$ removed is also not connected. On the other hand, this graph is what we get when we remove the two vertices $x, y$ from $G$. Thus $G_1^y$ not being connected implies that $xy$ is a ladder-bridge in $G$. Since this is already covered in Case~1, we may as well assume that $G_1^y$ is connected. A similar argument works for all other graphs, hence we may assume that the modified graphs are all simple connected bipartite cubic graphs. Hence the inductive hypothesis holds for every $G_j^y, G_j^z, G_j^w$. 

  \begin{figure}[t]
    \centering
    \begin{tikzpicture}
      \begin{scope}[shift={(12,0)},scale=0.7]
	\TCM
	\draw (b) -- (y) -- (f);
	\draw (a) -- (y);
	\draw (c) -- (z) -- (e);
	\draw (d) -- (z);
	\foreach \p in {(y), (z)} {
	  \draw[fill=white] \p circle (0.05);
	}
	\foreach \p in {(a), (b), (c), (d), (e), (f)} {
	  \draw[fill=black] \p circle (0.05);
	}
	\node at (-90:2) {$G_1^w$};
      \end{scope}
      \begin{scope}[shift={(12,-3.5)},scale=0.7]
	\TCM
	\draw (b) -- (y) -- (e);
	\draw (a) -- (y);
	\draw (c) -- (z) -- (f);
	\draw (d) -- (z);
	\foreach \p in {(y), (z)} {
	  \draw[fill=white] \p circle (0.05);
	}
	\foreach \p in {(a), (b), (c), (d), (e), (f)} {
	  \draw[fill=black] \p circle (0.05);
	}
	\node at (-90:2) {$G_2^w$};
      \end{scope}
      \begin{scope}[shift={(4,0)},scale=0.7]
	\TCM
	\draw (c) -- (z) -- (d);
	\draw (b) -- (z);
	\draw (e) -- (w) -- (f);
	\draw (a) -- (w);
	\foreach \p in {(w), (z)} {
	  \draw[fill=white] \p circle (0.05);
	}
	\foreach \p in {(a), (b), (c), (d), (e), (f)} {
	  \draw[fill=black] \p circle (0.05);
	}
	\node at (-90:2) {$G_1^y$};
      \end{scope}
      \begin{scope}[shift={(4,-3.5)},scale=0.7]
	\TCM
	\draw (c) -- (z) -- (d);
	\draw (a) -- (z);
	\draw (e) -- (w) -- (f);
	\draw (b) -- (w);
	\foreach \p in {(w), (z)} {
	  \draw[fill=white] \p circle (0.05);
	}
	\foreach \p in {(a), (b), (c), (d), (e), (f)} {
	  \draw[fill=black] \p circle (0.05);
	}
	\node at (-90:2) {$G_2^y$};
      \end{scope}
      \begin{scope}[shift={(8,0)},scale=0.7]
	\TCM
	\draw (a) -- (y) -- (b);
	\draw (c) -- (y);
	\draw (e) -- (w) -- (f);
	\draw (d) -- (w);
	\foreach \p in {(w), (y)} {
	  \draw[fill=white] \p circle (0.05);
	}
	\foreach \p in {(a), (b), (c), (d), (e), (f)} {
	  \draw[fill=black] \p circle (0.05);
	}
	\node at (-90:2) {$G_1^z$};
      \end{scope}
      \begin{scope}[shift={(8,-3.5)},scale=0.7]
	\TCM
	\draw (a) -- (y) -- (b);
	\draw (d) -- (y);
	\draw (e) -- (w) -- (f);
	\draw (c) -- (w);
	\foreach \p in {(w), (y)} {
	  \draw[fill=white] \p circle (0.05);
	}
	\foreach \p in {(a), (b), (c), (d), (e), (f)} {
	  \draw[fill=black] \p circle (0.05);
	}
	\node at (-90:2) {$G_2^z$};
      \end{scope}
    \end{tikzpicture}
    \caption{Case 3---modifying the graph $G$ to $G_j^y, G_j^z, G_j^w$}
    \label{fig:no-4-cycle-modification}
  \end{figure}
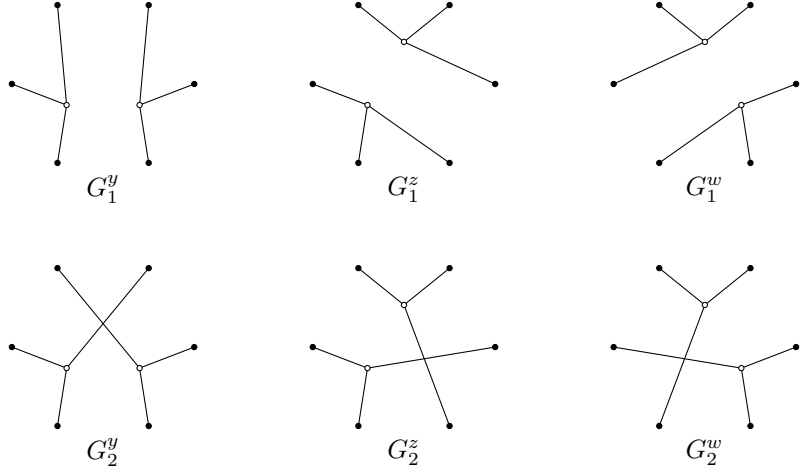

  Now we claim that given any $2$-factor of $G$, it can be modified to a $2$-factor of four of the graphs $G_j^v$ (with $v \in \{y,z,w\}$), in such a way that the $2$-factor of $G$ can be recovered from any of the modified $2$-factors. By symmetry, we only need to consider the case when the $2$-factor on the induced subgraph on $\{x,y,z,w,a,b,c,d,e,f\}$ is as in Figure~\ref{fig:no-4-cycle-resolution}. In this case, the $2$-factor of $G$ can be turned into $2$-factors of $G_1^z, G_1^w, G_2^y, G_2^w$ without changing it outside the depicted region. This shows that 
  \[
    4 \Fac(G) \le \sum_{v,j}^{} \Fac(G_j^v) \le 6 m_{n-1}
  \]
  and by Lemma~\ref{lem:inequalities-for-m}, \ref{itm:ineq-i}, we have $6 m_{n-1} < 4 m_n$. This concludes the proof. 
\end{proof}

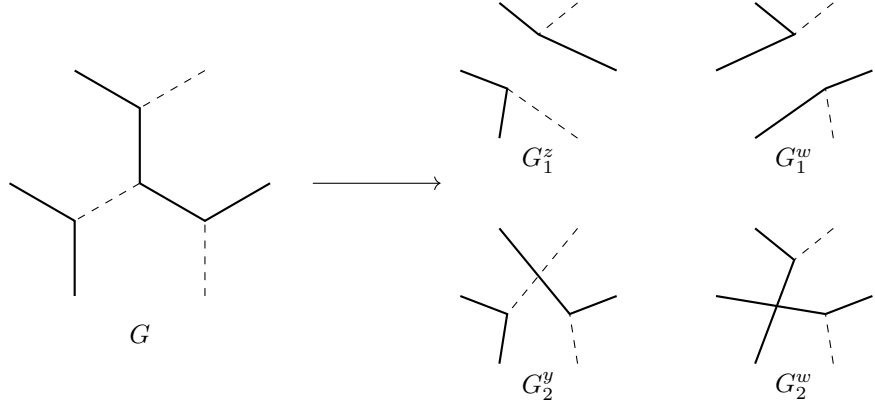
\begin{figure}[t]
  \centering
  \begin{tikzpicture}
    \begin{scope}
      \TC
      \draw[thick] (a) -- (y) -- (x) -- (z) -- (c);
      \draw[thick] (e) -- (w) -- (f);
      \draw[dashed] (b) -- (y);
      \draw[dashed] (d) -- (z);
      \draw[dashed] (x) -- (w);
      \node at (-90:2) {$G$};
    \end{scope}
    \draw[->] (2.3,0) -- (4,0);

    \begin{scope}[shift={(7,0)}]
      \begin{scope}[shift={(1.7,1.5)},scale=0.6]
	\TCM
	\draw[thick] (a) -- (y) -- (f);
	\draw[thick] (c) -- (z) -- (e);
	\draw[dashed] (b) -- (y);
	\draw[dashed] (d) -- (z);
	\node at (-90:2) {$G_1^w$};
      \end{scope}
      \begin{scope}[shift={(-1.7,1.5)},scale=0.6]
	\TCM
	\draw[thick] (a) -- (y) -- (c);
	\draw[thick] (e) -- (w) -- (f);
	\draw[dashed] (b) -- (y);
	\draw[dashed] (d) -- (w);
	\node at (-90:2) {$G_1^z$};
      \end{scope}
      \begin{scope}[shift={(1.7,-1.5)},scale=0.6]
	\TCM
	\draw[thick] (a) -- (y) -- (e);
	\draw[thick] (c) -- (z) -- (f);
	\draw[dashed] (b) -- (y);
	\draw[dashed] (d) -- (z);
	\node at (-90:2) {$G_2^w$};
      \end{scope}
      \begin{scope}[shift={(-1.7,-1.5)},scale=0.6]
	\TCM
	\draw[thick] (a) -- (z) -- (c);
	\draw[thick] (e) -- (w) -- (f);
	\draw[dashed] (b) -- (w);
	\draw[dashed] (d) -- (z);
	\node at (-90:2) {$G_2^y$};
      \end{scope}
    \end{scope}
  \end{tikzpicture}
  \caption{Case 3---modifying a $2$-factor of $G$ to $2$-factors of $G_1^z, G_1^w, G_2^y, G_2^w$}
  \label{fig:no-4-cycle-resolution}
\end{figure}

Let us now deduce Theorem~\ref{AA} from Theorem~\ref{AA-for-bipartite}. What we need is a clever lemma of Alon and Friedland \cite{AF08}. 

\begin{definition}
  Let $G$ be a simple cubic graph. We define a new graph $D(G)$ with vertices
  \[
    V(D(G)) = V(G) \times \{1,2\}
  \]
  and edges
  \[
    E(D(G)) = \{ (v,1) (w,2) : vw \in E(G) \}.
  \]
  The graph $D(G)$ is called the \textdef{bipartite double cover} of $G$. 
\end{definition}

Note that $D(G)$ is always a simple cubic bipartite graph if $G$ is a simple cubic graph. Moreover, if $G$ is connected and not bipartite, then $D(G)$ is connected. 

\begin{lemma}[Alon--Friedland \cite{AF08}] \label{lem:doubling-the-graph}
  Let $G$ be a simple cubic graph. Then
  \[
    \Fac(G)^2 \le \Fac(D(G)).
  \]
\end{lemma}

\begin{remark}
  Lemma~\ref{lem:doubling-the-graph} is the core ingredient used in the proof of the main result of Alon--Friedland \cite{AF08}. Even though their two-page paper does not explicitly state this lemma, the proof of the main reusult proceeds by first proving Lemma~\ref{lem:doubling-the-graph} and then combining it with the Bregman--Minc inequality. 
\end{remark}

We can finally prove Theorem~\ref{AA}. 

\begin{proof}[Proof of Theorem~\ref{AA}]
  For $n = 2$, there is only one simple cubic graph on $4$ vertices, namely $K_4$. Hence there is nothing to prove. 
  
  We now assume that $n \ge 3$. If $G$ is not bipartite, then $D(G)$ is connected and bipartite, so 
  \[
    \Fac(G)^2 \le \Fac(D(G)) \le m_{2n} < m_n^2
  \]
  by Lemma~\ref{lem:doubling-the-graph}, Theorem~\ref{AA-for-bipartite}, and Lemma~\ref{lem:inequalities-for-m}, \ref{itm:ineq-v}. This shows that $\Fac(G) < m_n$ if $G$ is not bipartite. The remaining case is when $G$ is bipartite, and this is covered by Theorem~\ref{AA-for-bipartite}. 
\end{proof}

\section{Proof of Theorem~\ref{BB}} \label{sec:quantum-field-theory}

Due to the connection with quantum physics, we use the bra-ket notation for
denoting vectors.  Consider the $2$-dimensional real vector space 
\[
  B=\mathbb{R} \lvert 0\rangle \oplus \mathbb{R} \lvert 1\rangle
  = \{a\lvert 0\rangle +b\lvert 1\rangle : a,b\in \mathbb{R} \}
\]
with basis $\lvert 0 \rangle$ and $\lvert 1 \rangle$. Its $n$-fold tensor
product $B^{\otimes n} = B \otimes \dotsb \otimes B$ is a $2^{n}$-dimensional
real vector space with basis
\[
  \lvert s_{1}\rangle \otimes \lvert s_{2}\rangle \otimes \cdots \otimes \lvert
  s_{n}\rangle 
\]
where $s_{1},s_{2},\dotsc ,s_{n} \in \{0,1\}$. 

Consider the element
\begin{equation*}
  \alpha = \lvert 0\rangle \otimes \lvert 0\rangle +\lvert 1\rangle \otimes
  \lvert 1\rangle \in B\otimes B,
\end{equation*}
and the linear map $\beta$ is given by
\begin{align*}
  \beta \colon B\otimes B\otimes B& \rightarrow
  \mathbb{R}, \\
  \lvert s_{1}\rangle \otimes \lvert s_{2}\rangle \otimes \lvert s_{3}\rangle
  & \mapsto
  \begin{cases}
    1 & \text{there are exactly two }0\text{ among }s_{1},s_{2},s_{3}, \\
    0 & \text{otherwise}.
  \end{cases}
\end{align*}

Given a cubic graph $G=(V,E)$ on $2n$ vertices, consider its incidence set 
\[
  \Phi =\{ (v,e): \text{vertex } v \text{ is incident to edge } e\} \subseteq
  V \times E.
\]
For each edge $e = uv \in E$, we have $(u,e),(v,e)\in \Phi$. To $e$ we assign
$\alpha \in B \otimes B$, where we interpret the two factors of $B$ as
corresponding to $(u,e)$ and $(v,e)$ each, and take their tensor product over
all $e$. The resulting element
\begin{equation*}
  \alpha^{\otimes E(G)}\in B^{\otimes \Phi} \cong B^{\otimes 6n}.
\end{equation*}
can be described as
\begin{equation*}
  \alpha ^{\otimes E(G)}=\sum_{\mathrm{col}}^{{}}\bigotimes_{(v,e)\in \Phi
  }\lvert s_{(v,e)}=\text{color of }e\rangle
\end{equation*}
where the sum is over all $2^{3n}$ colorings of the edges of $G$ by $0$ and $1$.

Each vertex is incident with three edges. Applying the linear map $\beta $
to each vertex and taking their tensor product, we obtain a linear map
\begin{equation*}
  \beta^{\otimes V(G)} \colon B^{\otimes \Phi} \to \mathbb{R}.
\end{equation*}
In terms of bases, it can be described as
\begin{equation*}
  \bigotimes_{(v,e)\in \Phi }\lvert s_{(v,e)}\rangle \mapsto
  \begin{cases}
    1 &
    \begin{matrix}
      \text{for every }v\text{ with edges }e_{1},e_{2},e_{3},\text{ there are} \\
      \text{exactly two }e_{i}\text{ with }s_{(v,e_{i})}=0, 
    \end{matrix}
    \\
    0 & \text{otherwise}.
  \end{cases}
\end{equation*}
It follows that
\begin{equation}
  \beta^{\otimes V(G)}(\alpha ^{\otimes E(G)})=\Fac(G). \label{eqn:tensor-1}
\end{equation}

We now consider a new basis of $B$ given by 
\[
  \lvert x\rangle =\frac{1}{\sqrt{2}}(\lvert 0\rangle +\lvert 1\rangle ),\quad \lvert y\rangle =\frac{1}{\sqrt{2}}(-\lvert 0\rangle +\lvert 1\rangle). 
\]
Then we can write
\[
  \alpha = \lvert x \rangle \otimes \lvert x \rangle + \lvert y \rangle \otimes \lvert y \rangle
\]
with respect to this new basis. On the other hand, in this basis, values of the map $\beta$ are
\begin{align}
  \beta \colon& \lvert x\rangle \otimes \lvert x\rangle \otimes \lvert x\rangle
  \mapsto \frac{3}{2\sqrt{2}},\quad \lvert x\rangle \otimes \lvert x\rangle
  \otimes \lvert y\rangle \mapsto -\frac{1}{2\sqrt{2}},  \notag \\
  & \lvert x\rangle \otimes \lvert y\rangle \otimes \lvert y\rangle \mapsto -
  \frac{1}{2\sqrt{2}},\quad \lvert y\rangle \otimes \lvert y\rangle \otimes
  \lvert y\rangle \mapsto \frac{3}{2\sqrt{2}}. \label{eqn:tensor-2}
\end{align}
The equation \eqref{eqn:tensor-1} holds in this new basis as well. Using \eqref{eqn:tensor-2}, we calculate
\begin{align*}
  \beta ^{\otimes V(G)}(\alpha ^{\otimes E(G)})& =\sum_{x,y\text{-col}} 
  \biggl(\frac{3}{2\sqrt{2}}\biggr)^{\#\text{ of homogeneous vert.}}\biggl(%
  \frac{-1}{2\sqrt{2}}\biggr)^{\#\text{ of colorful vert.}} \\
  & =\biggl(\frac{9}{8}\biggr)^{n}\sum_{x,y\text{-col}}^{{}}\biggl(-\frac{1}{3}%
  \biggr)^{\#\text{of colorful vert.}} \\
  &=\frac{1}{2^{3n}}\sum_{x,y\text{-col}}^{{}}(-3)^{\#\text{ of homogeneous vert.}}
\end{align*}
where the sum is over all $2^{3n}$ colorings of edges by $x$ and $y$, a
homogeneous vertex is a vertex with three edges of the same color, and a
colorful vertex is a vertex that is not homogeneous. This finishes the proof of
Theorem~\ref{BB}. 

We note that we could have made a different change of basis in the above proof,
instead of $\lvert x \rangle$ and $\lvert y \rangle$. The same method applied
to different bases produces other similar-looking formulas. We have chosen the
particular basis $\lvert x \rangle$ and $\lvert y \rangle$ because the end
result has an particularly elegant formulation. 

\section{Proof of Theorem~\ref{CC}} \label{sec:many-cycles}

It has been shown in \cite{Hor14} that $\Psi (r)\geq 2^{r-1} + 6 \cdot
2^{\frac{r-1}{2}}-O(r)$. The graph $MC_k$, defined below, improves the
exponential error term.

For $k \ge 3$, we define the cubic graph $MC_k$ as follows. There are $2k$ vertices
\[
  V(MC_k) = \{ x_i, y_i : 1 \le i \le k \}.
\]
Starting with these vertices, we add the edges $x_i x_{i+1}$ and $y_i y_{i+1}$ for $1 \le i \le k$, so that there are two $k$-cycles. Here, the indices are taken modulo $k$; $x_{i+k} = x_i$ and $y_{i+k} = y_i$. If $k$ is even, add the edges (which we colloquially call ``rungs'')
\[
  \{ x_{2i-1} y_{2i}, x_{2i} y_{2i-1} : 1 \le i \le k/2 \},
\]
and if $k$ is odd, add the rungs
\[
  \{ x_{2i-1} y_{2i}, x_{2i} y_{2i-1} : 1 \le i \le (k-1)/2 \} \cup \{ x_k y_k \}.
\]
It is then clear that $MC_k$ is always a simple connected cubic graph, with cyclomatic number given by $r(MC_k) = k+1$. See Figure~\ref{fig:cyclic-cross} for a picture of $MC_k$ for even $k$. 

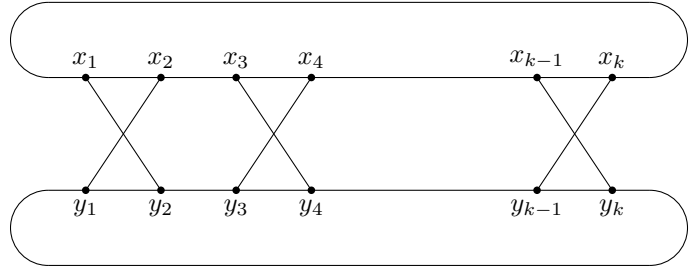
\begin{figure}[t]
  \centering
  \begin{tikzpicture}
    \draw (0,0) -- (8,0);
    \draw (0,1.5) -- (8,1.5);
    \foreach \x in {0,2,6} {
       \draw (\x+0.5,0) -- (\x+1.5,1.5);
       \draw (\x+1.5,0) -- (\x+0.5,1.5);
       \fill (\x+0.5,0) circle (0.05);
       \fill (\x+1.5,0) circle (0.05);
       \fill (\x+0.5,1.5) circle (0.05);
       \fill (\x+1.5,1.5) circle (0.05);
     }
     \node[above] at (0.5,1.5) {$x_1$};
     \node[above] at (1.5,1.5) {$x_2$};
     \node[above] at (2.5,1.5) {$x_3$};
     \node[above] at (3.5,1.5) {$x_4$};
     \node[above] at (6.5,1.5) {$x_{k-1}$};
     \node[above] at (7.5,1.5) {$x_{k}$};
     \node[below] at (0.5,0) {$y_1$};
     \node[below] at (1.5,0) {$y_2$};
     \node[below] at (2.5,0) {$y_3$};
     \node[below] at (3.5,0) {$y_4$};
     \node[below] at (6.5,0) {$y_{k-1}$};
     \node[below] at (7.5,0) {$y_{k}$};
     \draw (0,0) arc (90:270:0.5);
     \draw (0,1.5) arc (270:90:0.5);
     \draw (8,0) arc (90:-90:0.5);
     \draw (8,1.5) arc (-90:90:0.5);
     \draw (0,-1) -- (8,-1);
     \draw (0,2.5) -- (8,2.5);
  \end{tikzpicture}
  \caption{The graph $MC_k$ for even $k$}
  \label{fig:cyclic-cross}
\end{figure}

\begin{theorem}
  For $k \geq 3$, the number of cycles in $MC_k$ is
  \[
    \begin{cases} 2^k + (k + \frac{1}{2}) 2^{\frac{k}{2}} - \frac{3}{2}k & \text{if } k \text{ is even}, \\ 2^k + (k + \frac{7}{2}) 2^{\frac{k-1}{2}} - \frac{1}{2}(3k+5) & \text{ if } k \text{ is odd}. \end{cases}
  \]
  In particular, for $r \ge 4$, 
  \begin{align*}
    \Psi(r) &\ge 2^{r-1} + (r - \tfrac{1}{2}) 2^{\frac{r-1}{2}} - \tfrac{3}{2}(r-1) \text{ for } r \text{ odd}, \\
    \Psi(r) &\ge 2^{r-1} + (r + \tfrac{5}{2}) 2^{\frac{r-2}{2}} - \tfrac{1}{2}(3r+2) \text{ for } r \text{ even}.
  \end{align*}
\end{theorem}

\begin{proof}
  We only prove the statement for $k$ even, and leave the proof for $k$ odd to the reader. Given a cycle $C$, we first observe that the parity of
  \[
    w_i = \# (C \cap \{x_{2i} x_{2i+1}, y_{2i} y_{2i+1} \})
  \]
  is independent of $1 \le i \le k/2$. This is because the edges $x_{2i} x_{2i+1}$, $y_{2i} y_{2i+1}$, $x_{2i+2} x_{2i+3}$, $y_{2i+2} y_{2i+3}$ are the only edges that connect a vertex in the subset $\{ x_{2i+1}, x_{2i+2}, y_{2i+1}, y_{2i+2} \} \subseteq V(MC_k)$ and its complement, hence $C$ uses an even number of these edges. 

  We first count the number of cycles $C$ such that $w_i$ are all odd. Then either $x_{2i} x_{2i+1}$ is in $C$ and $y_{2i} y_{2i+1}$ is not in $C$, or $x_{2i} x_{2i+1}$ is not in $C$ and $y_{2i} y_{2i+1}$ is in $C$. There are in total $2^{k/2}$ ways to make this choice. For each choice, there are again $2^{k/2}$ ways of completing the set $C \cap \{ x_{2i} x_{2i+1}, y_{2i} y_{2i+1} : 1 \le i \le k/2 \}$ to all of $C$. This shows that the total number of cycles $C$ with $w_i$ odd is
  \[
    (\# \text{ of cycles } C \text{ with } w_i \text{ odd}) = 2^{k/2} \cdot 2^{k/2} = 2^{k}.
  \]

  We now count the number of cycles $C$ such that $w_i$ are all even. In this case, either $w_i = 0$ or $w_i = 2$. We observe that the $i$ such that $w_i = 2$ must form an interval modulo $n$. More precisely, there exist $1 \le s \le n$ and $0 \le d \le n$ such that
  \[
    w_s = w_{s+1} = \dotsb = w_{s+d-1} = 2, \quad w_{s+d} = w_{s+d+1} = \dotsb = w_{s+n-1} = 0.
  \]

  When $d = 0$, the number of ways of completing to a cycle $C$ is $k/2$, since the possible cycles are the $4$-cycles $x_{2i-1} x_{2i} y_{2i-1} y_{2i}$. If $1 \le d \le k/2-1$, then for each choice of $s$ there are $2^{d+1}$ ways of completing it to a cycle. If $d = k/2$, we only need to choose whether $x_{2i-1} x_{2i}$ and $y_{2i-1} y_{2i}$ are in $C$ or $x_{2i-1} y_{2i}$ and $y_{2i-1} x_{2i}$ are in $C$. For $C$ to be a $2k$-cycle instead of two $k$-cycles, the number of $x_{2i-1} y_{2i}$ in $C$ must be odd, hence there are $2^{\frac{k}{2}-1}$ cycles. Then
  \[
    (\# \text{ of cycles } C \text{ with } w_i \text{ even}) = \frac{k}{2} + \sum_{d=1}^{n-1} n \cdot 2^{d+1} + 2^{\frac{k}{2}-1} = \frac{k}{2} \cdot (2^{\frac{k}{2}+1}-3) + 2^{\frac{k}{2}-1}
  \]

  Hence, in aggregate, the number of cycles in $MC_k$ is
  \[
    2^{k} + \Bigl( \frac{k}{2} + \frac{1}{4} \Bigr) 2^{\frac{k}{2}+1} - \frac{3}{2} k.
  \]
  The stated lower bound on $\Psi(r)$ follows directly from $r(MC_k) = k+1$. 
\end{proof}

\section*{Acknowledgment}
\addcontentsline{toc}{section}{Acknowledgment}

The authors are grateful to Noga Alon for discussions on the main result of
this paper. 

\bibliographystyle{plain}
\bibliography{factors.bib}
\addcontentsline{toc}{section}{References}

\end{document}